\documentclass[a4j,11pt]{article}
\usepackage{amsmath}
\usepackage{amssymb}
\usepackage{amsthm}
\usepackage{comment}
\usepackage{graphicx, color} 
\usepackage{mathptmx}
\usepackage{setspace}
\usepackage{wrapfig}
\usepackage[all,knot,cmtip]{xy}
\xyoption{arc}
\entrymodifiers={+!!<0pt,\fontdimen22\textfont5>} 

{\theoremstyle{definition}
\newtheorem{definition}{Definition}[section]
\newtheorem{example}[definition]{Example}
}
\newtheorem{theorem}[definition]{Theorem}
\newtheorem{lemma}[definition]{Lemma}

\newtheorem{corollary}[definition]{Corollary}
{\theoremstyle{remark}
\newtheorem{remark}[definition]{Remark}
}

\makeatletter

\def\@paper{Cyclic Cohomology Groups of Some Self-similar Sets}
\def\id#1{\def\@id{#1}}
\def\adviser#1{\def\@adviser{#1}}

\def\@maketitle{
\
\vspace{10mm}
\begin{center}
{\Large \@paper \par} 
\vspace{10mm}
{\Large \@date\par} 
{\large \@author}

{\Large \@id \par} 

\end{center}
\par\vskip 1.5em
}

\makeatother

\title{Rational Homotopy Theory and Model category}
\date{}
\adviser{}

\author{Takashi Maruyama}
\id{}

\begin{document}

\newcommand{\Addresses}{{
  \bigskip
  \footnotesize

  \textsc{Graduate School of Mathematics, Nagoya University, Nagoya, Japan}\par\nopagebreak
  \textit{E-mail address} \texttt{m12047d@math.nagoya-u.ac.jp}
}}

\maketitle\thispagestyle{empty}
\begin{center}
\end{center}
\thispagestyle{empty}


\begin{abstract}
We define a variant of the Young integration on some kinds of self-similar sets which are called {\it cellular self-similar sets}. This variant is an analogue of the Young integration defined on the unit interval. We give the criteria of the variant on cellular self-similar sets, and also show that the variant is a cyclic $1$-cocycle of the algebra of complex-valued $\alpha$-H\"{o}lder continuous functions on the cellular self-similar sets. This suggests that the cocycle is a variant of currents. 
\end{abstract}
\section{Introduction}

Fractal sets introduced by Mandelbrot \cite{man} are complex-behaved spaces difficult to analyse. For instance, for the Cantor sets, their Hausdorff dimensions are different, although all of them are topologically isomorphic. Since the Hausdorff dimension is an invariant of fractal sets which is stable under bi-Lipschitz transformations, it is hard to say that the (co)homology theories which are homotopy invariant can capture deeper topological quantities of fractal sets.

Connes introduced cyclic cohomology theory \cite{con1}, which turns out to be a generalisation of the de Rham homology theory. He proposed Quantised calculus in \cite{con2} and exploits the Dixmier trace as a non-smooth analogue of the integration on manifolds. Namely, he applied it to the Cantor sets and succeeded to recover their Minkowski contents as the value of a certain Dixmier trace. This result suggests that cyclic cohomology theory could be one of suitable methods to study fractal sets. 

Another approach to analyse some fractal sets was proposed by Moriyoshi and Natsume \cite{mn}. For the Sierpinski gasket $SG$, they exploit the algebra $C^{1}(SG)$ of Lipschitz functions on $SG$ to construct a cyclic $1$-cocycle $\phi$ of $C^{1}(SG)$. They also show that, when Lipschitz functions are seen as $1$-H\"{o}lder continuous functions, the regularity $\alpha = 1$ of $C^{1}(SG)$ for the well-definedness of $\phi$ can be reduced to the half of the Hausdorff dimension $\operatorname{dim}_{H}(SG)/2$.  

In order to define the cyclic $1$-cocycle $\phi$, Moriyoshi and Natsume exploit the Young integration on the unit interval $I$. It has following properties, which explain a reason why H\"{o}lder continuous functions are used and the Hausdorff dimension arises in the results of \cite{mn}: the Young integration on the unit interval was developed in \cite{Young}. This is a bilinear function from the product $W_{\alpha} \times W_{\beta}$ of the Wiener classes such that $\alpha + \beta > 1 = \dim_{H}(I)$ to complex numbers:
$$Y : W_{\alpha} \times W_{\beta} \rightarrow \mathbb{C}.$$
Especially, if we restrict the domain of $Y$ to the algebra $C^{\alpha}(I)$ of $\alpha$-H\"{o}lder continuous functions on $I$, the map $Y$ is well-defined for $2\alpha > 1 = \dim_{H}(I)$. We note that $C^{\alpha}(I)$ contains the algebra $C^{\infty}(I)$ of smooth functions on $I$, and then, for $f$, $g \in C^{\infty}(I)$ we further get
$$Y (f, g) = \int_{I}fdg.$$
In this sense, the Young integration may be considered as a generalisation of the integration of differential $1$-forms, and does make it possible to integrate some non-smooth functions. Moreover, for a Jordan curve $C$ composed of a finite number of unit intervals, the Young integration $Y$ along $C$ turns out to be a cyclic $1$-cocycle if $2 \alpha > 1 = \dim_{H}(C)$:
$$Y : C^{\alpha}(C) \times C^{\alpha}(C) \rightarrow \mathbb{C}.$$
All those results above motivate us to extend the cyclic $1$-cocycle $\phi$ of $C^{1}(SG)$ onto a certain class of fractal sets by exploiting the Young integration.
\newline

In this paper, we extend the cocycle of the Sierpinski gasket defined in \cite{mn} to a certain class of self-similar sets by exploiting the Young integration, and show that the cocycles can be applied to some examples. More detailed and precise statements are given as follows.

We first define cellular self-similar sets, the preliminary notions of which are given in Section 2.2 below. Cellular self-similar sets $K_{|X|}$ are self-similar sets that are based on linear cell complexes, and the unit interval is a prototype of cellular self-similar sets. The precise definition is as follows:
\begin{definition}[Definition 3.1] 
Let $|X|$ be a $2$-dimensional finite convex linear cell complex and  $\{F_{j}\}_{j \in S}$ a set of similitudes $F_{j} : |X| \rightarrow |X|$ indexed by a finite set $S$. We also let $ |X_{1}| = \bigcup_{j \in S} F_{j}(|X|)$. The triple $(|X|, S, \{F_{j}\}_{j \in S})$ is called a {\it cellular self-similar structure} if it satisfies
\begin{itemize}
\item[a)] $\partial|X| \subset \partial |X_{1}|$, and
\item[b)] $\operatorname{int}F_{i}(|X|) \cap \operatorname{int}F_{j}(|X|) = \emptyset $, for all $ i \neq j \in S$.
\end{itemize}
\end{definition}
Let $(|X|, S, \{F_{j}\}_{j \in S})$ be a cellular self-similar structure. Then $(|X|, S, \{F_{j}\}_{j \in S})$ yields a sequence $\{|X_{n}|\}_{n \in \mathbb{N}}$ of $2$-dimensional cell complexes, and, by Theorem 2.5 below, the sequence gives rise to the cellular self-similar set $K_{|X|}$ with respect to $(|X|, S, \{F_{j}\}_{j \in S})$. For every $n \in \mathbb{N}$, $|X_{n}|$ is subdivided into a simplicial complex $|X_{n}^{s}|$ by Lemma 1 of Chapter 1 in \cite{zeeman}. From this simplicial complex, we get 1-chains $b_{n}$, $I_{n}$, $o_{n}$ and $I_{n}\backslash I_{n-1} \in \tilde{S}_{1}(X_{n}^{s}; \mathbb{C})$ whose geometric incarnations are subspaces of $1$-skelton of $|X_{n}^{s}|$; see Section 3.1 for the details. 
   
On the other hand, the algebra $C^{\alpha}(|X_{n}^{s}|)$ of complex-valued $\alpha$-H\"{o}lder continuous functions defined on $|X_{n}^{s}|$ is a subspace of the function space $F^{0}(|X_{n}^{s}|; \mathbb{C}) = \{f : |X_{n}^{s}| \rightarrow \mathbb{C}\ \}$ as a $\mathbb{C}$-vector space. Therefore,  $C^{\alpha}(|X_{n}^{s}|)$ can generate a $\mathbb{C}$-vector space $C^{\alpha, 1}(|X_{n}^{s}|)$ with the differential and cup product of the Alexander-Spanier cochain complex \cite{spanier}, which is a subspace of $\operatorname{Hom}_{\mathbb{C}}(\tilde{S}_{1}(|X_{n}^{s}|), \mathbb{C})$; see Section 3.2 for the details. For $f$ and $g \in C^{\alpha}(K_{|X|})$ we have a cochain $f \smile \delta(g) - g \smile \delta(f) \in C^{\alpha, 1}(|X_{n}^{s}|)$ for any $n \in \mathbb{N}$, which is denoted by $\omega_{n}(f, g)$. Finally we set $\phi_{n}(f, g)$ as $\omega_{n}(f, g)(I_{n})$ and call the sequence $\{\phi_{n}(f, g)\}_{n \in \mathbb{N}}$ the cyclic quasi-1-cocycle of $f$ and $g$, the definition of which is given in Section 3.2. 

The first main theorem states that if $2 \alpha > \operatorname{dim}_{H}(K_{|X|})$,  we can define a bilinear map $\phi : C^{\alpha}(K_{|X|}) \times C^{\alpha}(K_{|X|}) \rightarrow \mathbb{C}$ by taking the limit of the cyclic quasi-$1$-cocycle $\{\phi_{n}(f, g)\}_{n \in \mathbb{N}}$. This implies that the bilinear map $\phi$ may be seen as a generalisation of the classical Young integration of the unit interval.
\begin{theorem}[Theorem 3.8, Existence theorem]
Let $(|X|, S, \{F_{j}\}_{j \in S})$ be a cellular self-similar structure with $\#S \geq 2$ and $K_{|X|}$ the cellular self-similar set with respect to $(|X|, S, \{F_{j}\}_{j \in S})$. We also let $C^{\alpha}(K_{|X|})$ be the algebra of $\alpha$-H\"{o}lder continuous functions on $K_{|X|}$. 
 If $2 \alpha > \operatorname{dim}_{H}(K_{|X|})$, then the cyclic quasi-1-cocycle $\{\phi_{n}(f,g)\}$ is a Cauchy sequence for any $f$, $g \in C^{\alpha}(K_{|X|})$ .
\end{theorem}
The map $\phi$ is originally defined by Moriyoshi and Natsume \cite{mn} for the algebra $C^{Lip}(SG)$ of complex-valued Lipschitz functions on the Sierpinski gasket $SG$, and the construction is based on the classical Young integration on the unit interval. The authors use the simplexes $I_{n}$ to prove the existence of the cyclic cocycle of the Sierpinski gasket. An obstacle to extend the construction to cellular self-similar sets is that, for each $n \in \mathbb{N}$, the lengths of $1$-simplices belonging in $|X_{n}|$ are not equal. The key technical ingredient to overcome the difficulty is the existence of $2$-dimensional simplicial complex $|K_{n,n+1}^{s}|$ whose boundary is a disjoint union of $\partial(|X_{n}|)$ and $\partial(|X_{n+1}|)$. By properties of cellular self-similar sets, we can prove that lengths of $1$-simplices of $|K_{n, n+1}|$ have an upper bound which tends to $0$ as $n \rightarrow \infty$. This property plays a crucial role to prove that $\{\phi_{n}(f, g)\}_{n \in \mathbb{N}}$ is a Cauchy sequence.

The proof of the above theorem immediately yields the following corollary. This proves that the bilinear map $\phi$ is a non-commutative representation of the Young integration.
\begin{corollary}[Corollary 3.10]
For any $f$, $g \in C^{\alpha}(K_{|X|})$ with $2 \alpha > \dim (K_{|X|})$, we have 
$$\phi(f, g) = 2 \int_{\partial |X|}^{Young} f dg\ = 2 \cdot\ ({\rm Young\ integration\ of\ }f{\rm \ and\ }g{\rm \ along\ }\partial|X|).$$
In particular, if $|X| \neq |X_{1}|$, for $1$ and $x := id \in C^{\alpha}(K_{|X|})$, we get  
$$\phi(1, x) = 2 \int_{\partial |X|}^{Young}  dx =2 \cdot ( {\rm length\ of}\ \partial |X|).$$ 
\end{corollary}

After we define $\phi$ of the cellular self-similar set $K_{|X|}$, we prove that $\phi$ is a cyclic $1$-cocycle of $C^{\alpha}(K_{|X|})$ and represents a nontrivial element in the first cyclic cohomology group $HC^{1}(C^{\alpha}(K_{|X|}))$. This theorem shows that $\phi$ may be seen as a non-commutative generalisation of the integration on manifolds.
\begin{theorem}[Theorem 3.11]
Under the assumption of the existence theorem{\rm :} 
\begin{itemize}
\item[a)] The bilinear map $\phi$ is a cyclic $1$-cocycle of $C^{\alpha}(K_{|X|})$. 
\item[b)] If $|X| \neq |X_{1}|$, the cocycle $\phi$ represents a non-trivial element $[\phi]$ in $HC^{1}(C^{\alpha}(K_{|X|}))$.
\end{itemize}
\end{theorem}
For the proof of the first statement, we need to use the Leibniz rule of the cup product defined on the Alexander-Spanier cochain complex. Corollary 1.3 immediately completes the proof of the second statement since $1 \otimes x$ represents an element in the Hochschild homology group.

By Theorem 1.4, we find that the cocycle $\phi$ has the following additional properties: $\phi$ can detect the Hausdorff dimensions of cellular self-similar sets and distinguish them by their dimensions. 
For instance, we get the cocycles $\phi$ of the Sierpinski gasket $SG$ and the Sierpinski carpet $SC$, whose thresholds of the well-definedness are different. Namely, their thresholds are $\dim_{H}(SG) = \log_{2}3$ and $\dim_{H}(SC) = \log_{3}8$. Since bi-Lipschitz transformations preserve the Hausdorff dimension, the cocycles can prove that $SG$ and $SC$ are not bi-Lipschitz homeomorphic. Moreover, if we have a bi-Lipschitz transformation between cellular self-similar sets $K_{|X|}$ and $K_{|X^{\prime}|}$, the algebra $C^{\alpha}(K_{|X|})$ is isomorphic to $C^{\alpha}(K_{|X^{\prime}|})$. Therefore, we further get the following property: the cocycle $\phi$ is invariant under bi-Lipschitz transformations. 
After the proof of the main results, we apply the results to some examples, and extend the cocycle to some variants of cellular self-similar sets. 
\section*{Conventions}
We assume that algebras have unit unless otherwise stated, and all base rings of algebras is the field  $\mathbb{C}$ of complex numbers. The Euclidean space $\mathbb{R}^{n}$ is endowed with the standard Euclidean metric.
\begin{itemize}
\item $\otimes = \otimes_{\mathbb{C}}$.
\item$\mathbb{Z}_{\geq 0} = \mathbb{N} \cup \{0\}$.
\item $\partial X$ : the boundary of a topological space $X$.
\item $C^{\alpha}(X)$ : the algebra of complex-valued $\alpha$-H\"{o}lder continuous functions on a metric space $X$ whose sum and multiplication are given by the pointwise sum and multiplication.
\item $C^{Lip}(X)$ : the algebra of complex-valued Lipschitz functions on a metric space $X$ whose sum and multiplication are given by the pointwise sum and multiplication.
\end{itemize}

\subsection*{Acknowledgements}
The author thanks his advisor Hitoshi Moriyoshi for his constant advice and encouragements. The author also thanks Tatsuki Seto for many discussions.\\

\section{Preliminaries}
In this section, we recall the key materials for the main theorems, the Young integral, self-similar sets.
\subsection{Young Integral}
We begin with a quick review of the Young integral basically following \cite{Young} except the slight changes of the notation. 

Let $I$ be the unit interval $[0, 1]$ and $f$, $g$ complex-valued functions defined on $I$. We make a subdivision $\chi$ of $I$
$$0 = x_{0} < x_{1} < \cdots < x_{n-1} < x_{n} = 1$$
and define
$$F(\chi) = \sum_{i = 1}^{n} f(x_{i})(g(x_{i}) - g(x_{i-1})).$$
Then $F(\chi)$ can be also written as 
$$F(\chi) = \sum_{0 < i \leq j \leq n} \delta(f)(x_{i-1}, x_{i}) \cdot \delta(g)(x_{j-1}, x_{j}) + f(0)(g(1) -g(0)).$$
Here $\delta(f)(x_{i-1}, x_{i})$ denotes $f(x_{i}) - f(x_{i-1})$, and this notation is similar to the differential of the Alexander-Spanier cohomology theory; see also Chapter 6.4 in \cite{spanier}. We also let $\alpha, \beta > 0$, and denote by $S_{\alpha, \beta}[0, 1] = S_{\alpha, \beta}[0, 1 ; f, g]$ the upper bound of 
$$\biggl(\sum_{i}|\delta(f)(x_{i-1}, x_{i})|^{\frac{1}{\alpha}}\biggr)^{\alpha} \biggl(\sum_{i}|\delta(g)(x_{i-1}, x_{i})|^{\frac{1}{\beta}}\biggr)^{\beta}$$
for every subdivision of $I$. Following lemmas of \cite{Young}, if $\alpha + \beta > 1$ and $\xi \in [0, 1]$ is a division point of $\chi$, we have 
$$\Bigl|F(\chi) - f(\xi)(g(1) - g(0))\Bigr| \leq (1 + \zeta(\alpha + \beta)) \cdot S_{\alpha, \beta}[0, 1],$$
where $\zeta(\alpha + \beta)$ denotes the zeta function of $\alpha + \beta$.

This inequality yields to a more general inequality for the sum associated to $\chi$: for the given subdivision $\chi$, let a point $x_{i-1} \leq \xi_{i} \leq x_{i}$ for each $i$, and applying this inequality for each interval $[x_{i-1}, x_{i}]$ and summing up, we get 
$$\biggl|F(\chi) - \sum_{i = 1}^{n} f(\xi_{i})(g(x_{i}) - g(x_{i-1}))\biggr| \leq \{1 + \zeta(\alpha + \beta) \} \cdot \sum_{i = 1}^{n} S_{\alpha, \beta}[x_{i-1}, x_{i}\ ;\ f, g].$$ 
Moreover if we have another subdivision $\chi^{\prime}$ of $I$ and subdivision points $x_{j-1} \leq \xi^{\prime}_{j} \leq x_{j}$, then
$$\biggl| \sum_{i = 1}^{n} f(\xi_{i})(g(x_{i}) - g(x_{i-1})) - \sum_{j = 1}^{m} f(\xi^{\prime}_{j})(g(x^{\prime}_{j}) - g(x^{\prime}_{j-1}))\biggr| \hspace{4cm}$$ $$\hspace{2.5cm} \leq \{1 + \zeta(\alpha + \beta) \} \cdot \biggl\{ \sum_{i = 1}^{n} S_{\alpha, \beta}[x_{i-1}, x_{i}\ ;\ f, g] + \sum_{j = 1}^{m} S_{\alpha, \beta}[x^{\prime}_{j-1}, x^{\prime}_{j}\ ;\ f, g] \biggr\}.$$
\begin{definition}
We say that the {\it Stieltjes integral}
$$\int_{0}^{1} f dg$$
{\it exists in the Riemann sense with the value $J$}, if there exist $J \in \mathbb{C}$ and a function $\epsilon_{\delta} > 0$ with respect to the variable $\delta > 0$ such that $\epsilon_{\delta} \rightarrow 0$ as $\delta \rightarrow 0$, and if all the segments $[x_{i-1}, x_{i}]$ of a subdivision $\chi$ have lengths less than $\delta > 0$, then 
$$\Bigl|\ J - \sum_{i}f(\xi_{i})(g(x_{i}) - g(x_{i-1}))\Bigr| < \epsilon_{\delta}.$$

\end{definition}
We observe that, for the integrability in the Riemann sense, it is sufficient that the difference of any of two sums of the formula $\displaystyle \sum_{i}f(\xi_{i})(g(x_{i}) - g(x_{i-1}))$ of Definition 2.1, for each of which the length of $[x_{i-1}, x_{i}]$ is less than $\delta$, is less than $\epsilon_{\delta}$.  By the inequality just before Definition 2.1, this is the case if for some $\alpha, \beta > 0$ such that $\alpha + \beta > 1$ we have 
$$ \sum_{i = 1}^{n} S_{\alpha, \beta}[x_{i-1}, x_{i}\ ;\ f, g] < \epsilon_{\delta}.$$
For the existence of the integrability, we define $W_{\alpha}(\delta)$ to be the set of functions such that the value $V_{\alpha}^{(\delta)}(f)$ defined below has an upper bound:
$$V_{\alpha}^{(\delta)}(f) = \sup_{|\chi| \leq \delta}\biggl\{ \Bigl(\sum_{i} |f(x_{i}) - f(x_{i-1})|^{\frac{1}{\alpha}}\Bigr)^{\alpha}\biggr\} < \infty.$$
Here $|\chi|$ denotes the maximum length of the intervals of $\chi$, and the supremum runs over all subdivisions $\chi$ such that $|\chi|$ is less than or equal to $\delta$. Finally we define the {\it Wiener class} $W_{\alpha}$ to be the set of functions $f$ such that $V_{\alpha}^{(\delta)}(f)$ with respect to the variable $\delta$ has an upper bound. 


\begin{theorem}[Theorem on Stieltjes integrability]
If $f \in W_{\alpha}$ and $g \in W_{\beta}$ where $\alpha, \beta > 0$ and $\alpha + \beta > 1$, have no common discontinuities, their Stieltjes integral exists in the Riemann sense.
\end{theorem}
The Wiener class $W_{\alpha}$ is closed under the pointwise sum and scalar multiplication for $0 < \alpha < 1$. Therefore, if we regard the integration as a function from $W_{\alpha} \times W_{\alpha}$ to $\mathbb{C}$, this function turns out to be a bilinear function. 
On the other hand, it is clear from the definition that the set $C^{\alpha}(I)$ of complex-valued $\alpha$-H\"{o}lder continuous functions defined on $I$ is a subset of $W_{\alpha}$. Moreover, $C^{\alpha}(I)$ is closed under the pointwise multiplication in addition to the pointwise sum and scalar multiplication. The integration restricted to $C^{\alpha}(I)$ is also referred to as the {\it Young integration}. 
\begin{remark}
The Young integration is a special case of the Riemann-Stieltjes integration.
\end{remark}


\subsection{Self-similar Sets and Hausdorff Dimension}
In this subsection we briefly recall the definition of self-similar sets and the Hausdorff dimension. This subsection basically follows \cite{kigami}. At the end of this subsection, we give some examples of self-similar sets. 
We first begin with the definition of some maps from a metric space $(X, d)$ to itself. 
\begin{definition} Let $(X, d)$ be a metric space.
\begin{itemize}
\item[a)] A map $F : X \rightarrow X$ is a {\it contraction} if there exists $0 < r \leq 1$ such that $d(F(x), F(y)) \leq r \cdot d(x, y)$ for any $x$, $y \in X$. The real number $r$ is called the {\it contraction ratio}.
\item[b)] A contraction $F: X \rightarrow X$ is a {\it similitude} if $d(F(x), F(y)) = r \cdot d(x, y)$ for any $x$, $ y \in X$. We call $r$ the {\it similarity ratio}.
\end{itemize}
\end{definition}
For a finite set $\{F_{j}\}_{j \in S}$ of contractions defined on a complete metric space, there exists a unique compact subspace that is characterised by $\{F_{j}\}_{j \in S}$. Here is the precise statement of the existence of self-similar sets:
\begin{theorem}Let $X$ be a complete metric space. We also let $S$ be a finite set and $F_{i} : X \rightarrow X$ contractions indexed by $S$. We call the triple $(X, S, \{F_{j}\}_{j \in S})$ an {\rm iterated function system} or {\rm IFS}. Then, 
there exists a unique non-empty compact subset $K_{X}$ of $X$ that satisfies
$$K_{X} = \bigcup_{j \in S} F_{j}(K_{X}).$$
\end{theorem}
The compact set $K_{X}$ is called the {\it self-similar set with respect to} $(X, S, \{F_{j}\}_{j \in S})$. 
\begin{remark} In some literature the terminology ``self-similar set" is used in a restricted sense. For instance, Hutchinson introduces the notion of ``self-similar set" for a finite set of similitudes \cite{hutch}, and self-similar sets defined in Theorem 2.5 are referred to as {\it attractors} or {\it invariant sets}; see Section 9.1 in \cite{fal}. We employ Hutchinson's definition of self-similar sets in the last section to define cellular self-similar sets, the definition of which is given in Section $3.1$ below.
\end{remark}
For later use, we include an outline of a proof of the above theorem. The proof is based on the following theorem.
\begin{theorem}[Contraction principle]
Let $(X,d)$ be a complete metric space and $F: X \rightarrow X$ a contraction with respect to the metric. Then there exists a unique fixed point of $F$, in other words, there exists a unique solution to the equation $F(x) = x$. Moreover if $x_{*}$ is the fixed point of $F$, then $\{F^{n}(a)\}_{n \geq 0}$ converges to $x_{*}$ for all $a \in X$ where $F^{n}$ is the $n$-th iteration of $F$.
\end{theorem}

Let $(X,d)$ be a metric space and $K(X)$ the set of non-empty compact subsets of $X$. We define the {\it Hausdorff metric} $\delta$ on $K(X)$ by $$\delta(A,B) = \operatorname{inf}\{r > 0 : U_{r}(A) \subset B\ {\rm and}\ U_{r}(B) \subset A\},$$
where $U_{r}(A) = \{x \in X\ :\ d(x, A) \leq r\}$.
\begin{lemma}
The pair $(K(X), \delta)$ forms a metric space. Moreover, if $X$ is complete, $(K(X), \delta)$ is also complete.
\end{lemma}
We now assume that the metric space $(X, d)$ is complete. Define $  F(A) = \bigcup_{j \in S} F_{j}(A)$ for $A \subset X$, and then $F : K(X) \rightarrow K(X)$ is a contraction with respect to the metric $\delta$. Therefore, by applying Theorem 2.7 to $(K(X), \delta)$ and $F$, we get the self-similar set $K_{X}$ with respect to $(X, S, \{F_{j}\}_{j \in S})$.

We next define the Hausdorff dimension, which plays a key role to define cyclic cocycles on cellular self-similar sets, the definition of which are given in Section $3$ below.
\begin{definition}
Let $(X,d)$ be a metric space. We also let $s > 0$ and $\delta > 0$. For any bounded set $A \subset X$, we define
$$\mathcal{H}_{\delta}^{s}(A) = \inf \biggl\{ \sum_{i \geq 1} \operatorname{diam}(E_{i})^{s}\ :\ A \subset \bigcup_{i \geq1} E_{i},\ \operatorname{diam}(E_{i}) \leq \delta \biggr\}.$$
Here the infimum runs over all the coverings $\{ E_{i} \}$ of $A$, which consist of sets, and $\operatorname{diam}(E_{i})$ denotes the diameter of $E_{i}$.
Also we define $$  \mathcal{H}^{s}(A) = \operatornamewithlimits{limsup}_{\delta \downarrow 0}\mathcal{H}_{\delta}^{s}(A),$$
and we call $\mathcal{H}^{s}$ the $s${\it -dimensional Hausdorff measure} of $(X, d)$. 
\end{definition}
\begin{remark}
The $s$-dimensional Hausdorff measure is a complete Borel measure.
\end{remark}
The measure detects a critical point of the given subset.
\begin{lemma} For any subset $E \subset X$, we have
$$\sup \ \{ s \in \mathbb{R}\ |\ \mathcal{H}^{s}(E) = \infty \} = \inf \ \{ s \in \mathbb{R}\ |\ \mathcal{H}^{s}(E) = 0\ \}. $$
\end{lemma}
\begin{definition}
The real number which satisfies Lemma 2.11 is called the {\it Hausdorff dimension} of $E$, and it is denoted by $\dim_{H}(E)$. 
\end{definition}

In general, it is difficult to calculate the Hausdorff dimension. Namely, the Hausdorff dimensions of a few self-similar sets have been computed. However, if we have a self-similar set $K_{X}$ with respect to an IFS $(X, S, \{F_{j}\}_{j \in S})$ such that contractions are similitudes and the similitudes have ``small" enough intersections, then there exists a useful way to compute the Hausdorff dimension of $K_{X}$.
\begin{theorem} {\rm \cite[Theorem II]{mor}} \ Let $X$ be a compact subspace in $\mathbb{R}^{n}$ and $\{F_{j} : \mathbb{R}^{n} \rightarrow \mathbb{R}^{n}\}_{j \in S}$ a finite set of similitudes indexed with a finite set $S$. Suppose that the self-similar set $K_{X}$ with respect to the IFS $(X, S, \{F_{j}\}_{j \in S})$ satisfies the {\rm open set condition}, i.e., 
there exists a bounded non-empty open set $O \subset \mathbb{R}^{n}$ such that
$$\bigcup_{j \in S}F_{j}(O) \subset O\ \ \ and\ \ \ F_{i}(O) \cap F_{j}(O) = \emptyset \ \ for\ any\ i \neq j \in S.$$
Then the Hausdorff dimension $\dim_{H}(K_{X})$ of the self-similar set $K_{X}$ is the unique real number $\alpha$ such that the following relation holds
$$\sum_{j \in S} r_{j}^{\alpha} = 1.$$
Here $r_{j}$ denotes the similarity ratio of $F_{j}$.
\end{theorem}

\begin{example}
In this example we give examples of self-similar sets and their dimensions. For later use, we explain contractions of each self-similar set and give an IFS $(X, S, \{F_{j}\}_{j \in S})$ that gives rise to the self-similar set. We also provide figures for each self-similar set, that correspond to $X$, $  F(X) (= \bigcup_{j \in S} F_{j}(X))$ and $F \circ F(X)$.\\

$\bullet$ {\bf Sierpinski gasket}\\
The Sierpinski gasket $SG$ is a well-known examplesof self-similar sets. Here are the first $3$ steps of the construction of the Sierpinski gasket:
\begin{center}
  \includegraphics[bb=0 0 1057 292, width = 140mm ]{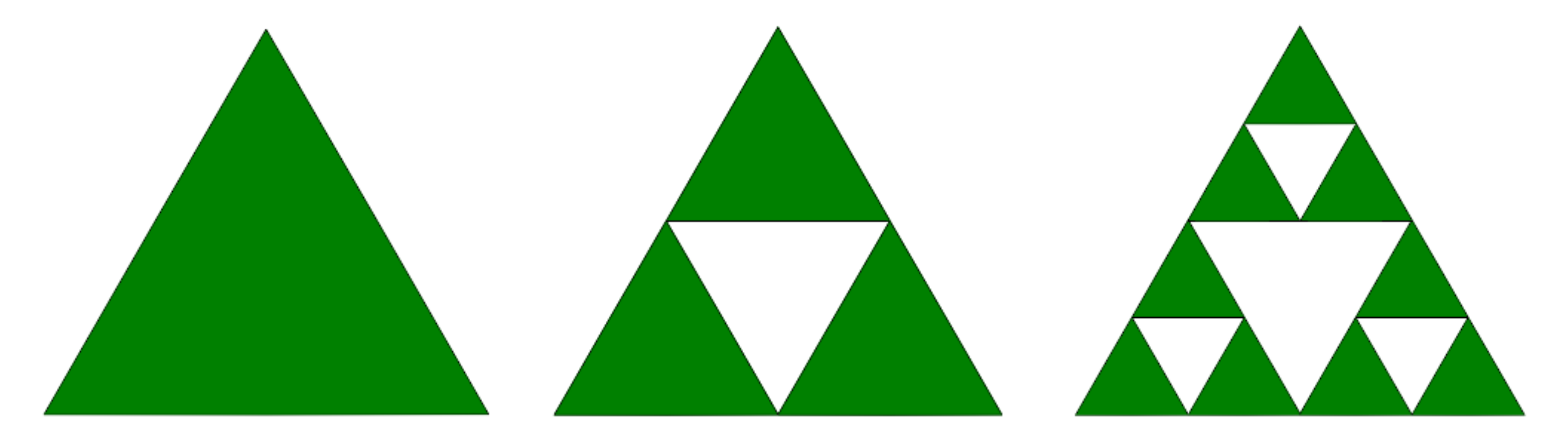}
\end{center}
The space of the left-hand side $X$ is an equilateral triangle in $\mathbb{R}^{2}$. In the centre we have $3$ equilateral triangles, the length of whose edges are a half of the ones of $X$. The similitudes $F_{1}$, $F_{2}$ and $F_{3}$ are defined by the $3$ triangles, and the similarity ratios of $F_{j}$ are $\frac{1}{2}$. The right-hand side is the space $F \circ F(X)$. Then, we get an IFS $(X, S = \{1, 2, 3\}, \{F_{j}\}_{j \in S})$, and it gives rise to $SG$. Moreover, $SG$ satisfies the open set condition. Namely, we can choose an open set $O = \operatorname{int}(X)$, and we find that $\bigcup_{j \in S} F_{j}(O) \subset O$ and $F_{i}(O) \cap F_{j}(O) = \emptyset$ for any $i \neq j \in S$. Therefore, the Hausdorff dimension of $SG$ is the root $\alpha$ given by the equation $\sum_{j \in S}$$(\frac{1}{2})^{\alpha} = 3 \cdot (\frac{1}{2})^{\alpha} = 1$, i.e., $\dim_{H}(SG) = \log_{2}3$.\\

$\bullet$ {\bf Pinwheel fractal}\\ 
The Pinwheel fractal $PW$ is a self-similar set which is modeled by the pinwheel tiling of the plane. There exist uncountably many pinwheel tilings, and therefore we have self-similar sets following them. 
\begin{center}
  \includegraphics[bb = 0 0 595 113, width = 130mm]{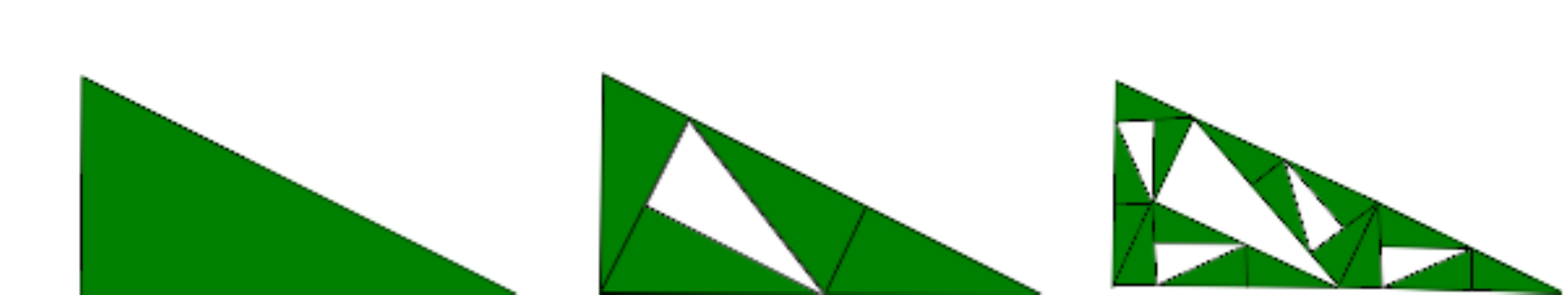}
\end{center}
The figure is one of the pinwheel fractals based on the most well-known pinwheel tiling of $\mathbb{R}^{2}$. 
The left triangle consists of $3$ edges whose lengths are $1$, $2$ and $\sqrt{5}$, and we have $4$ similitudes whose similarity ratios are $\frac{1}{\sqrt{5}}$. Therefore, we get an IFS $(X, S = \{1, \cdots, 4\}, \{F_{j}\}_{j \in S})$ which gives rise to $PW$. Since $PW$ satisfies the open set condition, the Hausdorff dimension of the pinwheel fractal is given by the root of the equation $  \sum_{j \in S}$$(\frac{1}{\sqrt{5}})^{\alpha} = 4 \cdot (\frac{1}{\sqrt{5}})^{\alpha} = 1$, i.e., $\dim_{H}(PW) = \log_{\sqrt{5}}4$.\\

$\bullet$ {\bf Infinite Sierpinski gasket}\\
Here we give a non-connected self-similar set based on the Sierpinski gasket. 
\begin{center}
  \includegraphics[bb = 0 0 1057 297, width = 140mm]{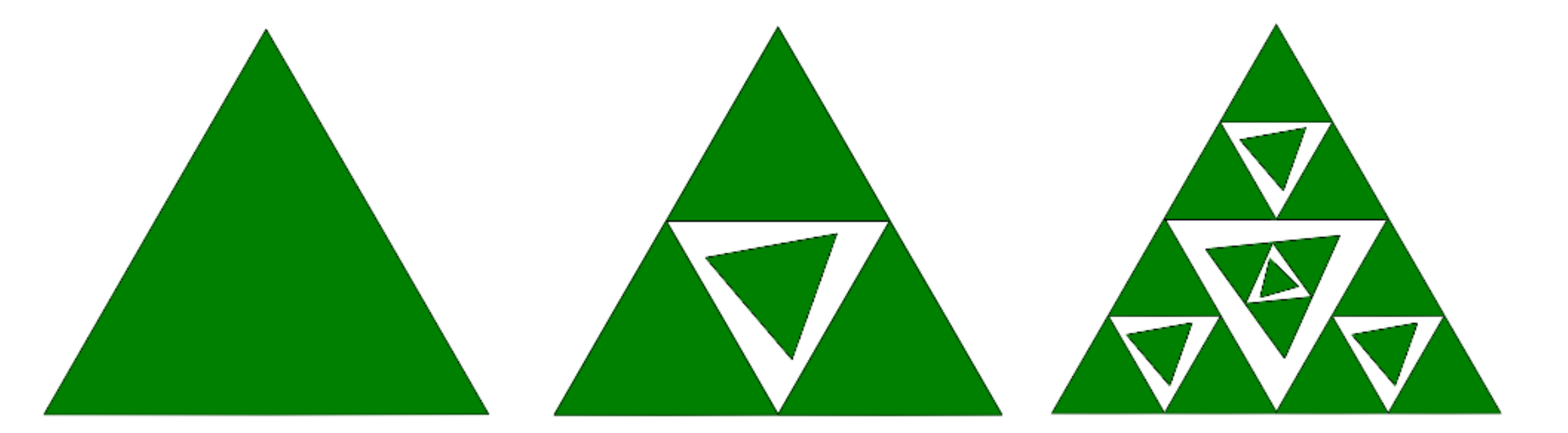}  
\end{center}
The row represents the first $3$ iterations of an IFS that consists of $4$ similitudes, one of which has the similarity ratio $\frac{1}{3}$ and the rest has $\frac{1}{2}$. The resulting cellular self-similar set consists of infinitely many countable connected components.

\end{example}
\section{Main Theorem}
In this section we define cyclic cocycles on certain subclass of self-similar sets and prove the main theorems. From now on, self-similar sets are assumed to be in $\mathbb{R}^{2}$.
\subsection{Cellular Self-similar Structures}
First we define the kinds of self-similar sets on which we define cyclic cocycles. 
\begin{definition} Let $|X|$ be a $2$-dimensional finite convex linear cell complex and  $\{F_{j}\}_{j \in S}$ a set of similitudes $F_{j} : |X| \rightarrow |X|$ indexed by a finite set $S$. We also let $  |X_{1}| = \bigcup_{j \in S} F_{j}(|X|)$. The triple $(|X|, S, \{F_{j}\}_{j \in S})$ is called a {\it cellular self-similar structure} if it satisfies
\begin{itemize}
\item[a)] $\partial|X| \subset \partial |X_{1}|$, and
\item[b)] $\operatorname{int}F_{i}(|X|) \cap \operatorname{int}F_{j}(|X|) = \emptyset $, for all $ i \neq j \in S$.
\end{itemize}
\end{definition}
By Theorem $2.5$ we have a unique self-similar set $K_{|X|}$ with respect to the cellular self-similar structure $(X, S, \{F_{j}\}_{j \in S})$ and we call $K_{|X|}$ the {\it cellular self-similar set with respect to} $(|X|, S, \{F_{j}\}_{j \in S})$. By construction, $K_{|X|}$ is a compact subset of $|X| \subset \mathbb{R}^{2}$.

\begin{lemma} Any cellular self-similar structure $(|X|, S, \{F_{j}\}_{j \in S})$ satisfies the open set condition. 
\end{lemma}
\begin{proof}
The lemma follows immediately from the definition of cellular self-similar structures.
\end{proof}
Let  $(|X|, S, \{F_{j}\}_{j \in S})$ be a cellular self-similar structure. For any $n \in \mathbb{N}$, we define a cell complex $|X_{n}|$ as follows: first, for $\omega = (j_{1}, \cdots, j_{n}) \in S^{\times n}$, we write 
$$F_{\omega} = F_{j_{1}} \circ \cdots \circ F_{j_{n}}.$$
We define $|X_{n}|$ by the following skelton filtration:
\begin{itemize}
\item $  sk_{0}(|X_{n}|) = \bigcup_{\omega \in S^{\times n}}F_{\omega}(sk_{0}(|X|))$,
\item $  sk_{1}(|X_{n}|) = \bigcup_{\omega \in S^{\times n}}F_{\omega}(sk_{1}(|X|))$,  
\item $  sk_{2}(|X_{n}|) = \bigcup_{\omega \in S^{\times n}}F_{\omega}(sk_{2}(|X|)) = \bigcup_{\omega \in S^{\times n}}F_{\omega}(|X|)$.
\end{itemize}
A $1$-cell in $|X_{n}|$ is defined to be the closure of a connected component in $sk_{1}(|X_{n}|) - sk_{0}(|X_{n}|)$. The definition of a cellular self-similar structure yields $$|X_{n+1}| =\bigcup_{j \in S}F_{j}\hspace{1pt}( \bigcup_{\omega \in S^{\times n}}F_{\omega}(|X|)) = \bigcup_{j \in S} F_{j}(|X_{n}|),$$
and therefore we have an inclusion map $i_{n, n+1} : |X_{n+1}| \hookrightarrow |X_{n}|$ for every $n \in \mathbb{Z}_{\geq 0}$. Moreover $K_{|X|}$ is written as the inverse limit of inclusion maps $\{i_{n, n+1} : |X_{n+1}| \hookrightarrow |X_{n}|\}$, that is, 
$$K_{|X|} = \bigcap_{n=1}^{\infty}|X_{n}|.$$
Therefore we also have a canonical inclusion map $i_{n} : K_{|X|} \hookrightarrow |X_{n}|$ for each $n \in \mathbb{Z}_{\geq 0}$.

For a $n \in \mathbb{N}$ and a $1$-cell $|\sigma|$ in $\partial |X_{n}|$, we define $E_{\sigma}^{n}
$ to be the set of $1$-cells of $|X_{n+1}|$ which are subspaces of $|\sigma|$. Then, we have $$|\sigma| = \bigcup_{|\tau| \in E_{\sigma}^{n}} |\tau|.$$
\begin{lemma}\label{lemma:1}
There exists $M \in \mathbb{N}$ that satisfy the following condition: for any $n \in \mathbb{N}$ and a $1$-cell $|\sigma|$ in $\partial |X_{n}|$ we have $\# E_{\sigma}^{n} \leq M$. 
\end{lemma}
\begin{proof} For every $1$-cell $|\sigma|$ in $\partial |X_{n}|$, there exists a unique $\omega \in S^{\times n}$ and a unique $1$-cell $|\tilde{\sigma}|$ in $F_{\omega}(|X|)$ such that $|\sigma| \subset |\tilde{\sigma}|$. Since $|X_{n+1}|$ is obtained by replacing each $2$-cell $F_{\omega}(|X|)$ by $  F_{\omega}(|X_{1}|) = F_{\omega}(\bigcup_{j \in S} F_{j}(|X|))$, $|\tilde{\sigma}|$ is subdivided by at most $\# S$ $2$-cells. This completes the proof of the lemma.
\end{proof}

Now, since $|X_{n}|$ is a convex linear cell complex, we can associate an abstract simplicial complex $X^{s}_{n}$ by employing a lemma in \cite{zeeman}:
\begin{lemma} {\rm \cite[Chapter I, Lemma 1]{zeeman}}
A convex linear cell complex can be subdivided into a simplicial complex without introducing any more vertices.  
\end{lemma}
For any simplicial complex $|X^{s}_{n}|$ and $p \geq 0$ we define $S_{p}(X_{n}^{s})$ to be a set of $(p+1)$-tuples of points of $sk_{0}(X_{n}^{s})$ such that $(p+1)$ vertices are contained in a simplex of $X_{n}^{s}$, that is, 
$$S_{p}(X_{n}^{s}) = \Bigl\{ (x_{0}, \cdots, x_{p}) \in sk_{0}(X_{n}^{s})^{\times (p+1)}\ |\ {\rm there\ exists\ a}\ p{\rm \mathchar`-simplex}\ \sigma \in X_{n}^{s}\ s.t.\ x_{i} \in \sigma \ {\rm for}\ \forall i \ \Bigr\}.$$
We also define face maps $\sigma_{i} : S_{p}(X^{s}_{n}) \rightarrow S_{p-1}(X^{s}_{n})$ for $0 \leq i \leq p$, and the pair $(S_{*}(X^{s}_{n}), \sigma_{i})$ forms a semi-simplicial set; see the definition \cite{ez}.
We note that, for $p \geq 1$, $S_{p}(X_{n}^{s})$ contains a {\it degenerate simplex} $(x_{0}, \cdots, x_{p})$, that is, a simplex $(x_{0}, \cdots, x_{p}) \in S_{p}(X^{s}_{n})$ such that there exist distinct indexes $i$ and $j$ such that $x_{i} = x_{j}$. 
Now, we define $\tilde{S}_{p}(X_{n}^{s}; \mathbb{C})$ to be the free $\mathbb{C}$-module generated by $S_{p}(X_{n}^{s})$ and a map $\tilde{\partial}_{p} : \tilde{S}_{p}(X_{n}^{s}; \mathbb{C}) \rightarrow \tilde{S}_{p-1}(X_{n}^{s}; \mathbb{C})$ by 
$$\tilde{\partial}_{p}(x_{0}, \cdots, x_{p}) = \sum_{j=0}^{p}(-1)^{j} \sigma_{i}(x_{0}, \cdots, x_{p}) = \sum_{j=0}^{p}(-1)^{j}(x_{0}, \cdots, \hat{x}_{j}, \cdots, x_{p}).$$ Then we have a commutative diagram:
\[\xymatrix{
\tilde{S}_{p}(X_{n}^{s}; \mathbb{C}) \ar[d]_{\pi} \ar[r]^{\tilde{\partial}_{p}} & \tilde{S}_{p-1}(X_{n}^{s}; \mathbb{C}) \ar[d]^{\pi} \\
C_{p}(X_{n}^{s}; \mathbb{C})  \ar[r]_{\partial_{p}}& C_{p-1}(X_{n}^{s}; \mathbb{C}), \\
}\]
where $C_{p}(X_{n}^{s}; \mathbb{C})$ is the $p^{\rm th}$ simplicial chain group of $X_{n}^{s}$ whose coefficient is $\mathbb{C}$, $\partial_{p}$ a simplicial boundary map and $\pi$ the quotient map. 
\begin{remark}
The chain map $\pi$ is a chain equivalence; see Theorem $8$ in Chapter 4.3 of \cite{spanier} for details.
\end{remark}

We now assign the counterclockwise orientation on each $2$-simplex in every $|X_{n}^{s}|$, and choose a basis $B_{n} = \{[\sigma]\}$ of $C_{2}(X_{n}^{s}; \mathbb{C})$ consisting of non-degenerate $p$-simplexes $\sigma$ in $X_{n}^{s}$. We assume that each element $[\sigma]$ of $B_{n}$ represents the counterclockwise orientation. 

Now, we define simplicial chains for every $n \in \mathbb{Z}_{\geq 0}$:
let 
$$c_{n} = \sum_{[\sigma] \in B_{n}} [\sigma] \in C_{2}(X_{n}^{s}; \mathbb{C}).$$ 
Then $\partial_{2}(c_{n}) \in C_{1}(X_{n}^{s}; \mathbb{C})$ is the sum of all $1$-simplices which lie on $\partial |X_{n}|$, and we can choose $s_{n} \in \pi^{-1}(c_{n})$ so that $s_{n}$ has no degenerate simplexes and each summand of $\tilde{\partial}_{2}(s_{n}) \in \tilde{S}_{1}(X_{n}^{s}; \mathbb{C})$ lies on $\partial|X_{n}|$. Now we define a boundary chain $ b_{n} \in \tilde{S}_{1}(X_{n}^{s}; \mathbb{C})$ by 
\begin{itemize}
\item $b_{n} = \tilde{\partial}_{2}(s_{n})$.
\end{itemize}
We first let $\epsilon(b_{n})$ be the subset of $1$-simplices in $S_{1}(X_{n}^{s})$ which are direct summands of $b_{n}$. Since any $\sigma \in \epsilon(b_{n})$ is non-degenerate, we can take the geometric realisation $|\sigma| \subset \partial |X_{n}^{s}|$. We also define a subset $\epsilon(o_{n}) \subset \epsilon(b_{n})$ by  
$$\epsilon(o_{n}) = \Bigl\{\sigma \in \epsilon(b_{n})\ |\ |\sigma| \subset \partial |X|\Bigr\}.$$
For each $\sigma \in \epsilon(o_{n})$, we have the sign of $\sigma$ in $b_{n}$ and denote it by $\operatorname{sgn}(\sigma)$. Define 
\begin{itemize}
\item $  o_{n} =\sum_{\sigma \in \epsilon(o_{n})} \operatorname{sgn}(\sigma) \cdot \sigma \in \tilde{S}_{1}(X_{n}^{s}; \mathbb{C})$, 
\item $I_{n} = b_{n} - o_{n} \in \tilde{S}_{1}(X_{n}^{s}; \mathbb{C})$.
\end{itemize}
Let $\epsilon(I_{n}) = \epsilon(b_{n}) \backslash \epsilon(o_{n})$. We also define $|\epsilon(I_{n})| = \bigcup_{\sigma \in \epsilon(I_{n})} |\sigma|$, and $\epsilon(I_{n}\backslash I_{n-1})$ in a manner similar to $\epsilon(o_{n})$:
$$\epsilon(I_{n}\backslash I_{n-1}) = \Bigl\{ \sigma \in \epsilon(b_{n})\ |\ |\sigma| \subset \overline{|\epsilon(I_{n})| \backslash |\epsilon(I_{n-1})|}\ \Bigr\}.$$
Finally we define a $1$-chain by 
\begin{itemize}
\item $  I_{n}\backslash I_{n-1} = \sum_{\sigma \in \epsilon(I_{n} \backslash I_{n-1})} \operatorname{sgn}(\sigma) \cdot \sigma \in \tilde{S}_{1}(X_{n}^{s}; \mathbb{C})$.
\end{itemize}
\begin{example}\ \\
For the Sierpinski gasket and the pinwheel fractal introduced in Section 2.2, we give spaces that represent $\epsilon(b_{0})$, $\epsilon(b_{1})$, $\epsilon(b_{2})$, and $\epsilon(I_{0})$, $\epsilon(I_{1})$, $\epsilon(I_{2})$. The first row corresponds to $\epsilon(b_{i})$, and the second corresponds to $\epsilon(I_{i})$. The dots in spaces denote the vertices of $1$-simplices, i.e., $0$-simplices.\\

\begin{figure}[htbp]
$\bullet$ Sierpinski gasket
\begin{center}
  \includegraphics[bb = 0 0 890 300, width = 130mm]{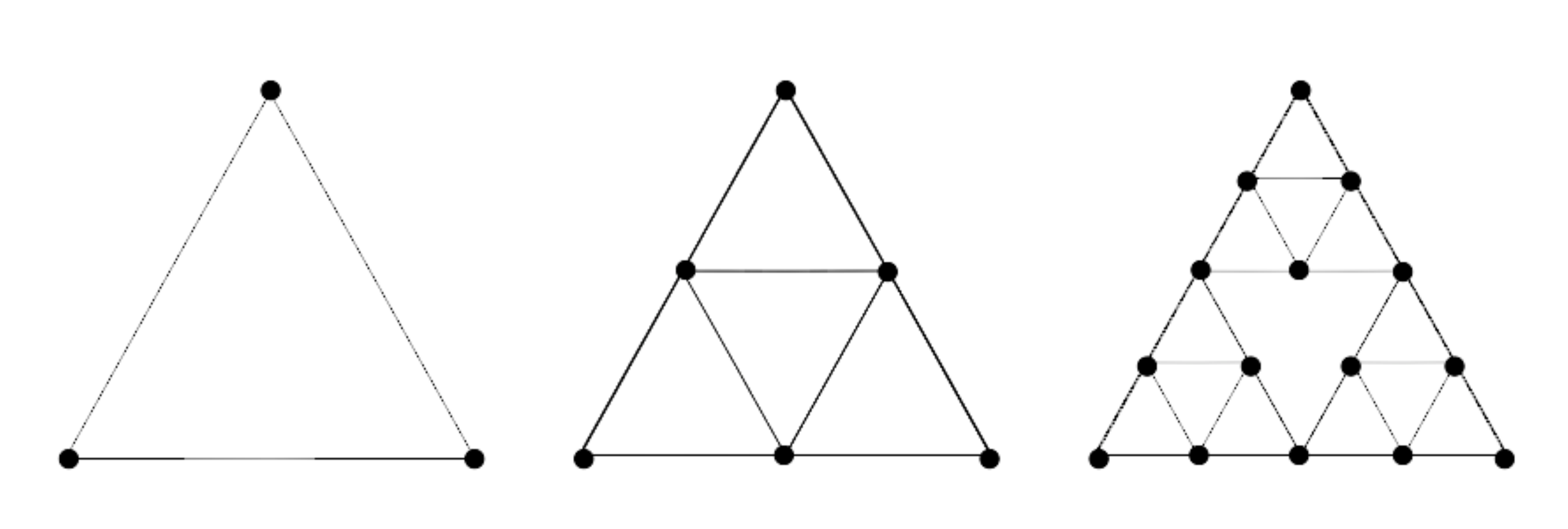}
\end{center}
\end{figure}
\begin{figure}[htbp]
\begin{center}
  \includegraphics[bb = 0 0 890 300, width = 130mm]{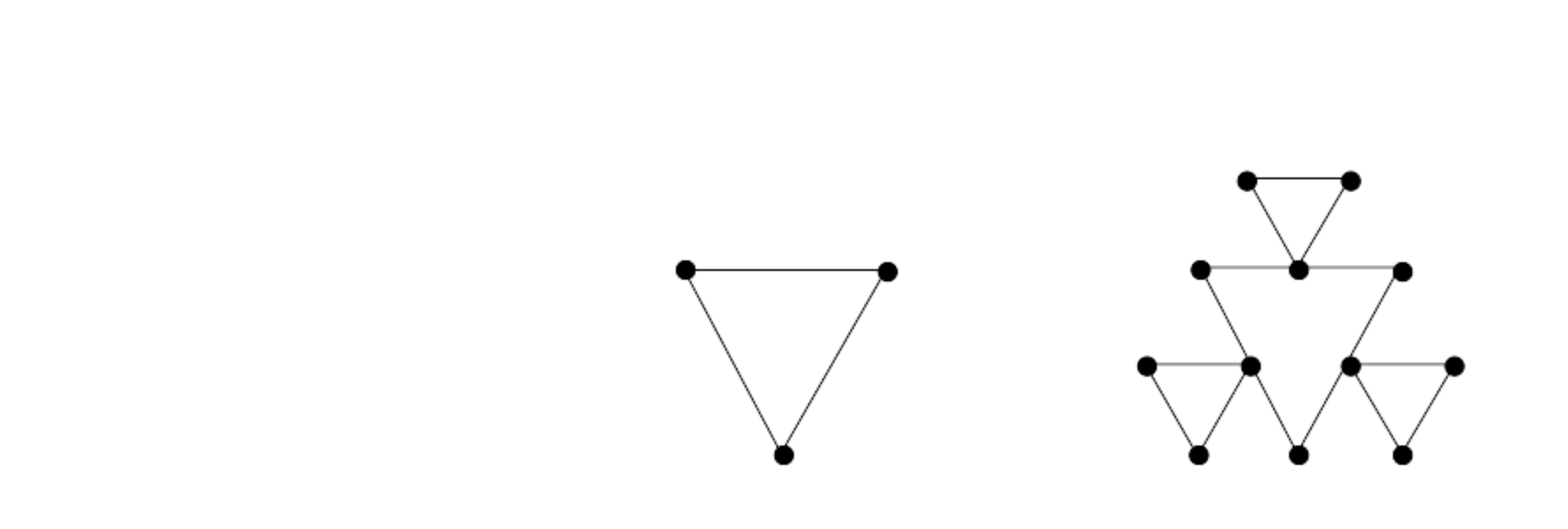}
\end{center}

\end{figure}

\begin{figure}[htbp]
$\bullet$ Pinwheel fractal
\begin{center}
  \includegraphics[bb = 0 0 655 128, width = 140mm]{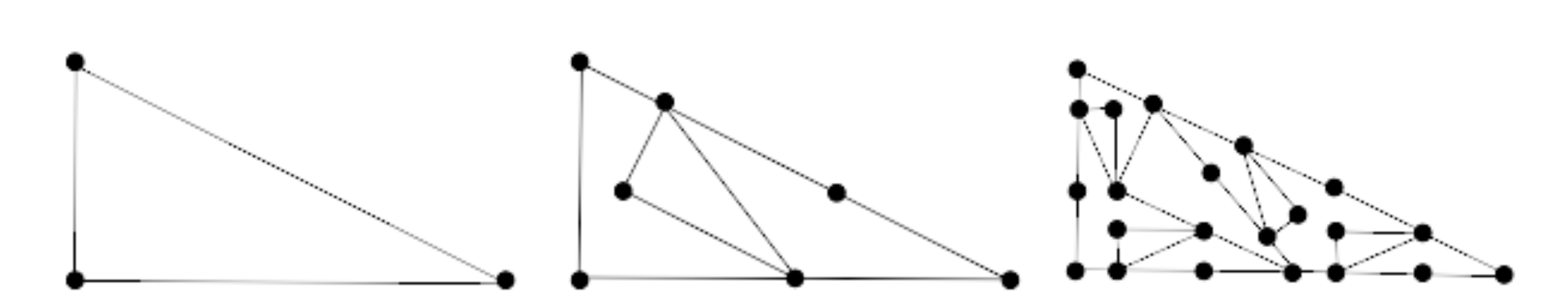}
  \includegraphics[bb = 0 0 655 128, width = 140mm]{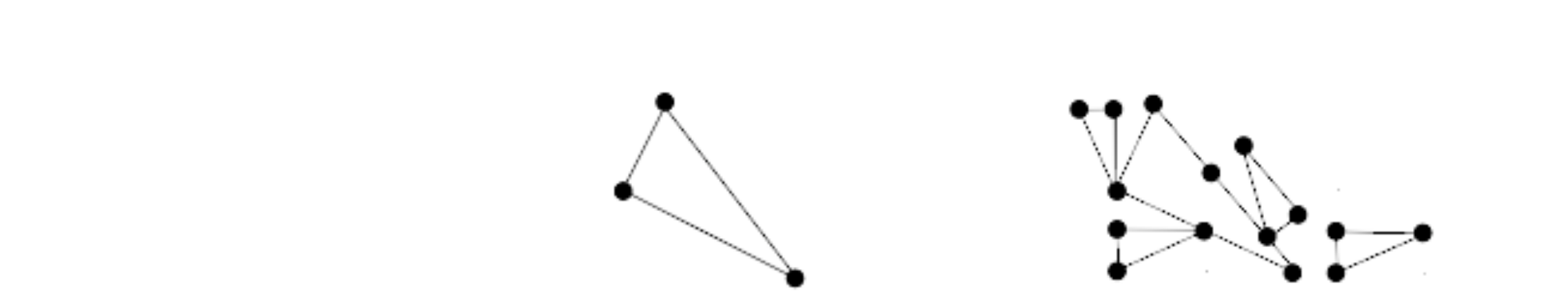}
\end{center}
\end{figure}

\end{example}

\ 

\

Next, for every $n \in \mathbb{Z}_{\geq 0}$, we define a $2$-dimensional cell complex $|K_{n,n+1}|$. 
For every $n \in \mathbb{Z}_{\geq 0}$ we endow
$$|X_{n,n+1}| = \overline{|X_{n}| - |X_{n+1}|}\ (= {\rm the\ closure\ of\ }|X_{n}| - |X_{n+1}|), $$ 
with a cell complex structure, whose structure is defined by the following skelton filtration:
\begin{itemize}
\item $sk_{0}(|X_{n, n+1}|) = sk_{0}(\partial |X_{n+1}|) \cap |X_{n, n+1}|$
\item $sk_{1}(|X_{n, n+1}|) = \partial |X_{n, n+1}|$
\item $sk_{2}(|X_{n, n+1}|) = |X_{n, n+1}|$
\end{itemize}
We also define a subspace $|K_{n,n+1}|$ in $\mathbb{R}^{3}$ to be 
$$|K_{n,n+1}| = [0,1] \times \partial|X_{n+1}| \cup \{1\} \times |X_{n,n+1}|.$$
We use $z$ as the variable of the first coordinate of $|K_{n,n+1}|$.
We now endow $|K_{n, n+1}|$ with a $2$-dimensional cell complex structure as follows: let $p_{1} : |K_{n,n+1}| \rightarrow |K_{n,n+1}||_{z = 1}$ be a projection defined by $p_{1}(t,x) = (1,x)$. We define

\begin{itemize}
\item $sk_{0}(|K_{n,n+1}|) = \{0\} \times sk_{0}(\partial|X_{n+1}|) \cup \{1\} \times (sk_{0}(\partial|X_{n}|) \cup sk_{0}(|X_{n, n+1}|) )$
\item $sk_{1}(|K_{n,n+1}|) = \{0\} \times sk_{1}(\partial|X_{n+1}|) \cup \{1\} \times (sk_{1}(\partial|X_{n}|) \cup sk_{1}(\partial|X_{n,n+1}|) \cup |E_{n, n+1}|)$
\item $sk_{2}(|K_{n,n+1}|) = |K_{n,n+1}|$.
\end{itemize}
Here, 
$$|E_{n, n+1}| = \Bigl\{(x,y)\ |\ x \in \{1\} \times sk_{0}(\partial|X_{n}|)\ {\rm or}\ x \in \{1\} \times sk_{0}(|X_{n, n+1}|),$$ $$\hspace{2.3cm} y \in \{0\} \times sk_{1}(\partial|X_{n+1}|)\ {\rm such\ that} \ p_{1}(x) = y \Bigr\}.$$By construction of $|K_{n,n+1}|$, we have 
$$
\partial |K_{n,n+1}| = \{0\} \times \partial|X_{n+1}| \cup \{1\} \times \partial|X_{n}|$$
as a cell complex in $\mathbb{R}^{3}$. By employing Lemma 3.4 again, the cell complex $|K_{n,n+1}|$ is subdivided into a $2$-dimensional simplicial complex $|K_{n,n+1}^{s}|$, and we may therefore choose chains $s_{n, n+1}$, $\tilde{s}_{n, n+1}$ and $\tilde{\tilde{s}}_{n, n+1} \in \tilde{S}_{2}(K_{n,n+1}^{s}; \mathbb{C})$ so that the chains consist of non-degenerate simplexes and 
$$\tilde{\partial}_{2}(s_{n, n+1}) = b_{n} - b_{n+1},\ \ \ 
\tilde{\partial}_{2}(\tilde{s}_{n, n+1}) = I_{n} - I_{n+1},\ \ \ 
\tilde{\partial}_{2}(\tilde{\tilde{s}}_{n, n+1}) = I_{n+1}\backslash I_{n}.$$
We define the sets $\epsilon(s_{n, n+1})$, $\epsilon(\tilde{s}_{n, n+1})$ and $\epsilon(\tilde{\tilde{s}}_{n, n+1})$ in a manner similar to the definition of $\epsilon(b_{n})$, and assume that $\tilde{s}_{n, n+1}$ and $\tilde{\tilde{s}}_{n, n+1}$ are  summands of $s_{n, n+1}$, in other words, 
$$\epsilon(\tilde{s}_{n, n+1}),\ \epsilon(\tilde{\tilde{s}}_{n, n+1}) \subset \epsilon(s_{n, n+1}).$$ 

By a closed cycle $z$ in $I_{n+1}\backslash I_{n}$ we mean a subset $z$ of $\epsilon(I_{n+1} \backslash I_{n})$ such that $\bigcup_{\sigma \in z} |\sigma|$ is homomorphic to $S^{1}$, and denote $  \bigcup_{\sigma \in z} |\sigma|$ by $|z|$. 
We also denote by $\operatorname{cyc}(I_{n+1} \backslash I_{n})$ the set of closed cycles in $I_{n+1} \backslash I_{n}$ and define $  \tilde{z} = \sum_{\sigma \in z} \operatorname{sgn}(\sigma) \cdot \sigma \in \tilde{S}_{1}(K_{n, n+1}^{s}; \mathbb{C})$ for $z \in \operatorname{cyc}(I_{n+1}\backslash I_{n})$. Then, for every closed cycle $z$ in $I_{n+1}\backslash I_{n}$, there exists a non-degenerate $2$-chain $\tilde{\tilde{s}}_{z} \in \tilde{S}_{2}(K_{n,n+1}^{s}; \mathbb{C})$ such that $\tilde{\partial}_{2}(\tilde{\tilde{s}}_{z}) = \tilde{z}$.

For $n = 0$ we define 
$$|\tilde{K}_{0,1}| = [0, 1] \times \partial(|X| - |X_{1}|) \cup \{1\} \times |X_{0, 1}|$$ 
and then $|\tilde{K}_{0, 1}|$ is written as 
$$|\tilde{K}_{0, 1}| = \bigcup_{z \in \operatorname{cyc}(I_{1} \backslash I_{0})} |\epsilon(\tilde{\tilde{s}}_{z})|$$
since $\partial(|X| - |X_{1}|) = |\epsilon(I_{1} \backslash I_{0})|$.
Moreover, since, 
for every $\omega \in S^{\times n}$, we have an inclusion map 
$i_{\omega} : \partial(|X| - |X_{1}|) \hookrightarrow F_{\omega}(\partial |X_{1}|)$, there exists a family $\{ \tilde{i}_{\omega}\}_{\omega \in S^{\times n}}$ of inclusion maps $\tilde{i}_{\omega} : |\tilde{K}_{0,1}| \hookrightarrow |K_{n,n+1}|$ such that $$\tilde{i}_{\omega}|_{z = 0} = i_{\omega}.$$ Finally we fix a subdivision of $|\tilde{K}_{0,1}|$ and 
assume that the subdivision of the images of the inclusion maps are given by the subdivision of $|\tilde{K}_{0,1}|$.




\subsection{Cyclic Quasi-$1$-cocycle}
In this subsection, we define a sequence of complex numbers for given H\"{o}lder continuous functions, that we call a cyclic quasi-$1$-cocycle. In order to define the sequence, we first recall a cochain complex which gives rise to one of the classical cohomology theories in algebraic topology, so called Alexander-Spanier cohomology theory; see Chapter 6 of \cite{spanier} for details.

Let $R$ be a ring. We also let $X$ be a set and $X^{(p+1)}$ the $(p+1)$-fold product of $X$. We define $F^{p}(X;R)$ to be the abelian group of functions from $X^{(p+1)}$ to $R$, whose sum is given by the pointwise sum. A coboundary homomorphism $\delta : F^{p}(X; R) \rightarrow F^{p+1}(X; R)$ is defined by
$$(\delta \phi)(x_{0}, \cdots, x_{p+1}) = \sum_{j=0}^{p+1}(-1)^{j}\phi(x_{0}, \cdots, \hat{x_{j}}, \cdots, x_{p+1}).$$
We also introduce the cup product on the complex $(F^{*}(X;R), \delta )$: for $\phi_{1} \in F^{p}(X; R)$ and $\phi_{2} \in F^{q}(X; R)$ the cup product $\phi_{1} \smile \phi_{2} \in F^{p+q}(X; R)$ is defined by
$$(\phi_{1} \smile \phi_{2})(x_{0}, \cdots, x_{p+q}) = \phi_{1}(x_{0}, \cdots, x_{p})\phi_{2}(x_{p}, \cdots, x_{p+q}).$$
The Leibniz rule holds for the cup product: for $\phi_{1} \in F^{p}(X; R)$ and $\phi_{2} \in F^{q}(X; R)$,

$$\delta(\phi_{1} \smile \phi_{2}) = \delta \phi_{1} \smile \phi_{2} + (-1)^{p} \phi_{1} \smile \delta \phi_{2}.$$


Now, we define a cochain subcomplex of $(F^{*}(X;R), \delta)$: we assume that $X$ is a metric space and $R$ the field of complex numbers $\mathbb{C}$. We also let $C^{\alpha}(X)$ be the algebra of complex-valued $\alpha$-H\"{o}lder continuous functions on $X$. Then, $C^{\alpha}(X)$ is a subalgebra of $F^{0}(X;\mathbb{C})$, and for each $p \in \mathbb{Z}_{ \geq 0}$ we define the submodule $C^{\alpha, p}(X)$ of $F^{p}(X;\mathbb{C})$ generated by $C^{\alpha}(X) \subset F^{0}(X;\mathbb{C})$ with the coboundary maps and the cup product. 

We now apply the construction for a cellular self-similar structure $(|X|, S, \{F_{j}\}_{j \in S})$: let $C^{\alpha}(K_{|X|})$ be the $\alpha$-H\"{o}lder continuous functions defined on $K_{|X|}$. For each $n \in \mathbb{N}$, we endow $sk_{0}(|X_{n}|)$ with the induced metric of $\mathbb{R}^{2}$. Since we have an inclusion map $j_{n} : sk_{0}(|X_{n}|) \hookrightarrow K_{|X|}$ for every $n \in \mathbb{Z}_{\geq 0}$, we have a commutative diagram of cochain complexes 
\[\xymatrix{
F^{p}(K_{|X|}; \mathbb{C}) \ar[r]^{j_{n}^{*}\ \ } & F^{p}(sk_{0}(|X_{n}^{s}|); \mathbb{C}) \\
C^{\alpha, p}(K_{|X|}) \ar@{^{(}-_>}[u] \ar[r]_{j_{n}^{*}\ \ }& C^{\alpha, p}(sk_{0}(|X_{n}^{s}|)) \ar@{^{(}-_>}[u] \\
}\]
The vector space $F^{p}(sk_{0}(|X_{n}^{s}|); \mathbb{C})$ is the set of complex-valued functions $\operatorname{Func}({S}_{p}(\Delta^{\# sk_{0}(|X_{n}^{s}|)}), \mathbb{C})$ defined on ${S}_{p}(\Delta^{\# sk_{0}(|X_{n}^{s}|)}) : = sk_{0}(|X_{n}^{s}|)^{\times p+1}$. In a manner similar to the definition of the face maps $\sigma_{i}$ of $S_{p}(X_{n}^{s})$, we define the face maps on $S_{*}(\Delta^{\# sk_{0}(|X_{n}^{s}|)})$, and then the pair $(S_{*}(\Delta^{\# sk_{0}(|X_{n}^{s}|)}), \sigma_{i})$ turns out to be a semi-simplicial set, the definition of which is in \cite{ez}, also known as the fundamental $\infty$-groupoid of $sk_{0}(|X_{n}^{s}|)$. Since the inclusion map $S_{*}(|X_{n}^{s}|) \hookrightarrow {S}_{*}(\Delta^{\# sk_{0}(|X_{n}^{s}|)})$ is a map of semi-simplicial sets, 
we therefore get the following commutative diagram: 
\[\xymatrix{
F^{p}(K_{|X|}; \mathbb{C}) \ar[r]^{j_{n}^{*}\ \ } & F^{p}(sk_{0}(|X_{n}^{s}|); \mathbb{C})  \ar[r]^{{\rm extend}\ \ \ \ \ \ \ \ \ \  }_{{\rm linearly}\ \ \ \ \ \ \ \ \ } & \operatorname{Hom}_{\mathbb{C}}(\tilde{S}_{p}(\Delta^{\# sk_{0}(|X_{n}|)}; \mathbb{C}), \mathbb{C}) \ar[d]^{{\rm restrict}} \\
C^{\alpha, p}(K_{|X|}) \ar@{^{(}-_>}[u] \ar[r]_{j_{n}^{*}\ \ }& C^{\alpha, p}(sk_{0}(|X_{n}^{s}|) \ar@{^{(}-_>}[u] \ar@{^{(}-_>}[u] \ar[r]_{r} & \operatorname{Hom}_{\mathbb{C}}(\tilde{S}_{p}(|X_{n}^{s}|; \mathbb{C}), \mathbb{C}). \\
}\]

Now we define $C^{\alpha, p}(|X_{n}^{s}|) = \operatorname{im}(r)$.
For any $f$, $g \in C^{\alpha}(K_{|X|})$ and $p = 1$, we have a $1$-cochain $\omega_{n}(f,g) = (f \smile \delta g) - (g \smile \delta f)$ in $C^{\alpha, 1}(|X_{n}^{s}|)$ for every $n \in \mathbb{N}$. As we define in Section 3.1, we also have $I_{n} \in \tilde{S}_{1}(|X_{n}^{s}|;\mathbb{C})$. For every $n \in \mathbb{N}$, we have a complex number $\omega_{n}(f, g)(I_{n})$ and denote it by $\phi_{n}(f,g)$. 

\begin{definition} Let $f$, $g \in C^{\alpha}(K_{|X|})$. We call the sequence $\{\phi_{n}(f,g)\}_{n \in \mathbb{N}}$ the {\it cyclic quasi-1-cocycle} for $f$ and $g$. 
\end{definition}

\subsection{Non-trivial Cyclic $1$-cocycles}
We prove the main results in this subsection. We refer the reader to \cite{con1, con2} for details of the Hochschild cohomology groups and the cyclic cohomology groups. 

\begin{theorem}[{\bf Existence theorem}]
Let $(|X|, S, \{F_{j}\}_{j \in S})$ be a cellular self-similar structure with $\#S \geq 2$ and $K_{|X|}$ the cellular self-similar set with respect to $(|X|, S, \{F_{j}\}_{j \in S})$. We also let $C^{\alpha}(K_{|X|})$ be the algebra of complex-valued $\alpha$-H\"{o}lder continuous functions on $K_{|X|}$. If $2 \alpha > \operatorname{dim}_{H}(K_{|X|})$, then the cyclic quasi-1-cocycle $\{\phi_{n}(f,g)\}_{n \in \mathbb{N}}$ is a Cauchy sequence for any $f$, $g \in C^{\alpha}(K_{|X|})$.
\end{theorem}
\begin{proof}
We first endow $|K_{n, n+1}|$ with a metric by $d((t,x), (t', x')) = |x-x'|_{\mathbb{R}^{2}}$. Let $f$, $g \in C^{\alpha}(K_{|X|})$. Since we have an inclusion map $sk_{0}(|X_{n+1}|) \hookrightarrow K_{|X|}$ for every $n \in \mathbb{Z}_{\geq 0}$, we can extend $f$ to $f_{n} \in C^{\alpha}(sk_{0}(|K_{n, n+1}|))$ so that $f_{n}(t,x) = f(x)$ for $(t, x) \in [0, 1] \times \partial |X_{n+1}^{s}|$. We also let, for $h$, $k \in C^{\alpha}(sk_{0}(|K_{n, n+1}|))$, $\omega_{n}(h, k) = (h \smile \delta k) - (k \smile \delta h)$ be a $1$-cochain in $C^{\alpha, 1}(|K_{n, n+1}^{s}|)$. Then, we have 
\begin{eqnarray} \nonumber
|\phi_{n}(f,g) - \phi_{n+1}(f,g)| &=& |\omega_{n}(f_{n}, g_{n})(I_{n} - I_{n+1})|\\ \nonumber
&=& |\omega_{n}(f_{n}, g_{n})(\tilde{\partial}_{2}(\tilde{s}_{n, n+1}))| \\ \nonumber
&\leq& |\omega_{n}(f_{n}, g_{n})(\tilde{\partial}_{2}(\tilde{s}_{n, n+1}))| + |\omega_{n}(f_{n}, g_{n})(\tilde{\partial}_{2}(s_{n, n+1} - \tilde{s}_{n, n+1}))| \\ \nonumber
&\leq& \sum_{\sigma \in \epsilon(s_{n, n+1})} |\omega_{n}(f_{n}, g_{n})(\tilde{\partial}_{2}(\sigma))| \\ 
&=&  \sum_{\sigma \in \epsilon(s_{n, n+1})} |(\delta f_{n} \smile \delta g_{n})(\sigma) - (\delta g_{n} \smile \delta f_{n})(\sigma)|.
\end{eqnarray} 
We note that every $\sigma \in \epsilon(s_{n, n+1})$ is given by $\sigma = (x, y, z)$ for some $x$, $y$, $z \in sk_{0}(|K_{n, n+1}|)$. Therefore, $(1)$ may be written as
\begin{eqnarray}\nonumber
(1) \ \ \ &=& \sum_{(x,y,z) \in \epsilon(s_{n, n+1})} |(\delta f_{n} \smile \delta g_{n})(x,y,z) - (\delta g_{n} \smile \delta f_{n})(x,y,z)| \\ \nonumber 
&=&  \sum_{(x,y,z) \in \epsilon(s_{n, n+1})} \Bigl|(f_{n}(y)-f_{n}(x)) (g_{n}(z) - g_{n}(y)) - (g_{n}(y) - g_{n}(x))(f_{n}(z) - f_{n}(y))\Bigr| \\ 
&\leq& \sum_{(x,y,z) \in \epsilon(s_{n, n+1})} 2 \cdot c_{f} \cdot c_{g} |y - x|^{\alpha} |z - y|^{\alpha}, 
\end{eqnarray}
where $c_{f}$ and $c_{g}$ are the H\"{o}lder constants of $f$ and $g$ respectively. 

We now define a map to estimate the term $(2)$. For any $\sigma \in \epsilon(s_{n, n+1}) \backslash \epsilon(\tilde{\tilde{s}}_{n, n+1})$ there exists a unique $\omega = (j_{1}, \cdots, j_{n}) \in S^{\times n}$ such that $p_{1}(|\sigma|) \subset \partial F_{\omega}(|X|)$. We therefore have a map $\rho : \epsilon(s_{n, n+1}) \backslash \epsilon(\tilde{\tilde{s}}_{n, n+1}) \rightarrow S^{\times n}$, and define $\tilde{S}^{\times n}$ to be $\operatorname{im}(\rho)$.  We note that, by Lemma \ref{lemma:1}, there exists $M \in \mathbb{N}$ such that $\# \rho^{-1}(\omega) < M$ for any $\omega \in \tilde{S}^{\times n}$. 
Moreover, since $p_{1}(|\sigma|) \subset \partial F_{\omega}(|X|)$ we have an inequality
$$\operatorname{diam}(|\sigma|) = \operatorname{diam}(p_{1}(|\sigma|)) \leq r_{j_{1}} \cdot \cdots \cdot r_{j_{n}} \cdot d_{K_{|X|}},$$
where $(j_{1}, \cdots, j_{n}) = \omega \in \tilde{S}^{\times n}$, $r_{j}$ are the similarity ratios of $F_{j}$ and $d_{K_{|X|}}$ is the diameter of $K_{|X|}$.

On the other hand, we let $L = \# \operatorname{cyc}(I_{1} \backslash I_{0})$ be the number of closed cycles in $I_{1} \backslash I_{0}$. At the $(n+1)$-step, for every $\omega \in S^{\times n}$, there exist $L$ closed cycles in $F_{\omega}(\bigcup_{j \in S} F_{j}(|X|)) = F_{\omega}(|X_{1}|)$. We recall that for every closed cycle $z$ in $I_{n+1}\backslash I_{n}$ there is a $2$-chain $\tilde{\tilde{s}}_{z} \in \tilde{S}_{2}(K^{s}_{n, n+1})$ such that $\epsilon(\tilde{\tilde{s}}_{z}) \subset \epsilon(\tilde{\tilde{s}}_{n, n+1})$ and $\tilde{\partial}_{2}(\tilde{\tilde{s}}_{z}) = \tilde{z}$; see also Section 3.1. Therefore, $\tilde{\tilde{s}}_{n, n+1}$ may be written as
$$\tilde{\tilde{s}}_{n, n+1} = \sum_{\omega \in S^{\times n}} \sum_{1 \leq i \leq L} \tilde{\tilde{s}}_{\omega,\hspace{1pt} z_{i}}.$$
We also recall from Section $3.1$ that for every $\omega \in S^{\times n}$ we have an inclusion map $\tilde{i}_{\omega} : |\tilde{K}_{0,1}^{s}| \hookrightarrow |K_{n,n+1}|$ and 
$$\operatorname{im}(\tilde{i}_{\omega}) \hspace{0.4cm} = \bigcup_{z \in \operatorname{cyc}(I_{n+1} \backslash I_{n})\ {\rm s.t.}\ |z| \subset F_{\omega}(|X_{1}|)} |\epsilon(\tilde{\tilde{s}}_{ z})|.$$
Therefore, since the subdivision of the images $\operatorname{im}(\tilde{i}_{\omega})$ are induced by the subdivision of $|K_{0,1}^{s}|$, 
we may define
$$\overline{M} \hspace{0.3cm}= \sup_{z \in \operatorname{cyc}(I_{n+1} \backslash I_{n})}\{ \# \epsilon(\tilde{\tilde{s}}_{z}) \} \hspace{0.3cm} = \sup_{z \in \operatorname{cyc}(I_{1} \backslash I_{0})}\{ \# \epsilon(\tilde{\tilde{s}}_{z}) \}.$$

From these arguments, $(2)$ is now decomposed into two parts:
\begin{eqnarray} \nonumber
(2)\ \ \ &=& \sum_{(x,y,z) \in \epsilon(s_{n, n+1}) \backslash \epsilon(\tilde{\tilde{s}}_{n, n+1})} 2 \cdot c_{f} \cdot c_{g} |y - x|^{\alpha} |z - y|^{\alpha} \\ \nonumber
&\ & \hspace{3cm} + \sum_{(x,y,z) \in \epsilon(\tilde{\tilde{s}}_{n, n+1})} 2 \cdot c_{f} \cdot c_{g} |y - x|^{\alpha} |z - y|^{\alpha} \\ \nonumber
&\leq& \sum_{(j_{1}, \cdots, j_{n}) \in \tilde{S}^{\times n}} 2 \cdot c_{f} \cdot c_{g} \cdot \# \rho^{-1}(\omega) \cdot (r_{j_{1}}^{2 \alpha}\  \cdots\  r_{j_{n}}^{2\alpha} \cdot d_{K_{|X|}}^{2 \alpha}) \\ \nonumber
&\ & \hspace{3cm} + \sum_{(j_{1}, \cdots, j_{n}) \in S^{\times n}} \sum_{1 \leq i \leq L} 2 \cdot c_{f} \cdot c_{g} \cdot \# \epsilon(\tilde{\tilde{s}}_{z_{i}}) \cdot (r_{j_{1}}^{2 \alpha}\  \cdots\  r_{j_{n}}^{2\alpha} \cdot d_{K_{|X|}}^{2 \alpha})  \nonumber
\end{eqnarray}
\begin{eqnarray} \nonumber
\ &\leq& \sum_{(j_{1}, \cdots, j_{n}) \in S^{\times n}} 2 \cdot c_{f} \cdot c_{g} \cdot M \cdot (r_{j_{1}}^{2 \alpha}\  \cdots\  r_{j_{n}}^{2\alpha} \cdot d_{K_{|X|}}^{2 \alpha}) \\ \nonumber
&\ & \hspace{3cm} + \sum_{(j_{1}, \cdots, j_{n}) \in S^{\times n}}  2 \cdot c_{f} \cdot c_{g} \cdot L \cdot \overline{M} \cdot (r_{j_{1}}^{2 \alpha}\  \cdots\  r_{j_{n}}^{2\alpha} \cdot d_{K_{|X|}}^{2 \alpha}) \\ \nonumber
&=& 2 \cdot c_{f} \cdot c_{g} \cdot d_{K_{|X|}}^{2 \alpha} \cdot (M + L \cdot \overline{M}) \cdot (\sum_{j \in S} r_{j}^{2 \alpha})^{n}. \nonumber
\end{eqnarray}
We denote $2 \cdot c_{f} \cdot c_{g} \cdot d_{K_{|X|}}^{2 \alpha} \cdot (M + L \cdot \overline{M})$ by $K$, and then we have 
\begin{eqnarray}\nonumber
|\phi_{n+k}(f,g) - \phi_{n}(f,g)| &\leq& \sum_{1 \leq i \leq k} |\phi_{n+i}(f,g) - \phi_{n+i-1}(f,g)| \\ \nonumber
&\leq& \sum_{1 \leq i \leq k} K \cdot (\sum_{j \in S} r_{j}^{2 \alpha})^{n+i-1} \\ 
&=&  K \cdot (\sum_{j \in S} r_{j}^{2 \alpha})^{n} \cdot \sum_{1 \leq i \leq k} (\sum_{j \in S} r_{j}^{2 \alpha})^{i-1}. 
\end{eqnarray}
Since we assume that $2 \alpha > \operatorname{dim}_{H}(K_{|X|})$ and $\operatorname{dim}_{H}(K_{|X|})$ is computed by the formula in Theorem 2.13, and therefore the term $  (\sum_{j \in S} r_{j}^{2 \alpha})$ is less than $1$, and the term $  \sum_{1 \leq i \leq k} (\sum_{j \in S} r_{j}^{2 \alpha})^{i-1}$ converges to a finite value as $k$ tends to $\infty$. Therefore, we have 
$$(3) \ \ \ \leq \ \ \  K \cdot \sum_{i = 1}^{\infty} (\sum_{j \in S} r_{j}^{2 \alpha})^{i-1} \cdot (\sum_{j \in S} r_{j}^{2 \alpha})^{n},$$
and the right hand side also converges to $0$ as $n$ tends to $\infty$. This completes the proof of the theorem.
\end{proof}
From now on, we assume that $2 \alpha > \operatorname{dim}_{H}(K_{|X|})$ and define a bilinear map 
$$\phi : C^{\alpha}(K_{|X|}) \times C^{\alpha}(K_{|X|}) \rightarrow \mathbb{C}$$ 
by $\phi(f,g) = \lim_{n \rightarrow \infty}\phi_{n}(f,g)$.

\begin{lemma}
The map $\phi$ is independent of the choice of $I_{n}$.
\end{lemma}
\begin{proof}
In order to check the mentioned property of the bilinear map $\phi : C^{\alpha}(K_{|X|}) \times C^{\alpha}(K_{|X|}) \rightarrow \mathbb{C}$, we have to show that the cyclic quasi-$1$-cocycle converges to the same value regardless of the choice of $I_{n}$ which represents the given orientation. Let $I_{n}$, $I'_{n} \in \pi^{-1}([I_{n}])$ such that $|\epsilon (I_{n})| = |\epsilon (I_{n}')|$ and $\phi_{n}'(f, g) = (f \smile \delta g)(I_{n}') - (g \smile \delta f)(I_{n}')$. Then there exists a $2$-dimensional simplicial complex $J_{n}$ such that $|J_{n}| = |\epsilon(I_{n})| \times [0, 1]$, and we may choose $\hat{s}_{n} \in \tilde{S}_{2}(J_{n};\mathbb{C})$ such that $\tilde{\partial}_{2}(\hat{s}_{n}) = I_{n} - I'_{n}$. We endow $|J_{n}|$ with a metric similar to the metric on $|K_{n,n+1}|$, and then we have 
\begin{eqnarray}\nonumber
|\phi_{n}(f,g) - \phi'_{n}(f,g)| 
&=& |\delta \omega_{n}(f_{n}, g_{n})(\hat{s}_{n})| \\ \nonumber
&\leq& 2 \sum_{(x, y, z) \in \epsilon(\hat{s}_{n})} c_{f} \cdot c_{g} \cdot |y - x|^{\alpha} \cdot |z - y|^{\alpha} \\ \nonumber
&\leq& 2 \sum_{(x, y) \in \epsilon(I_{n})} 2 \cdot c_{f} \cdot c_{g} \cdot |y - x|^{2\alpha} \\ \nonumber
&\leq& 2 \sum_{(j_{1}, \cdots, j_{n}) \in \tilde{S}^{\times n}} 2 \cdot c_{f} \cdot c_{g} \cdot d_{K_{|X|}}^{2 \alpha} \cdot r_{j_{1}}^{2 \alpha} \cdot \cdots \cdot r_{j_{n}}^{2 \alpha} \\ \nonumber
&\leq& 4 \cdot c_{f} \cdot c_{g} \cdot d_{K_{|X|}}^{2 \alpha} \cdot (\sum_{j \in S} r_{j}^{2 \alpha})^{n} \\ \nonumber
&\rightarrow& 0, \ \ \ {\rm as}\ n \rightarrow \infty. \nonumber
\end{eqnarray}
This completes the proof of the well-definedness of $\phi$.
\end{proof}

Based on the proof of Theorem 3.8, we can prove the following corollary.
\begin{corollary} For any $f$, $g \in C^{\alpha}(K_{|X|})$, we have 
$$\phi(f, g) = 2 \int_{\partial |X|} f dg\ = 2 \cdot\ ({\rm Young\ integral\ along\ \partial|X|}).$$
In particular, for $1$ and $x := id \in C^{\alpha}(K_{|X|})$,  
$$\phi(1, x) = 2 \int_{\partial |X|}  dx =2 \cdot ( {\rm length\ of}\ \partial |X|).$$ 
\end{corollary}
\begin{proof}
By the construction of the cyclic quasi-$1$-cocycle of $f$, $g \in C^{\alpha}(K_{|X|})$, we have
$$\phi_{n}(f, g) = \omega_{n}(f, g)(I_{n}) = -\omega_{n}(f, g)(o_{n}) + \omega_{n}(f, g)(b_{n}).$$
The proof of Theorem 3.8 yields directly that the sequence $\{ \omega_{n}(f, g)(b_{n})\}_{n \in \mathbb{Z}_{\geq 0}}$ converges to $0$ if $2 \alpha > \operatorname{dim}_{H}(K_{|X|})$.  Since $\{ \omega_{n}(f, g)(o_{n}) \}_{n \in \mathbb{Z}_{\geq 0}}$ provides the Young integration along $\partial |X|$, which is the finite union of closed segments, we get the mentioned equalities.
\end{proof}

\begin{theorem} Under the assumption of Theorem 3.8 {\rm :} 
\begin{itemize}
\item[a)] The bilinear map $\phi$ gives rise to a cyclic $1$-cocycle of $C^{\alpha}(K_{|X|})$. 
\item[b)] If $|X| \neq |X_{1}|$, the cocycle $\phi$ represents a non-trivial element $[\phi]$ in $HC^{1}(C^{\alpha}(K_{|X|}))$.
\end{itemize}
\end{theorem}
\begin{proof}
We have a linear map $\phi : C^{\alpha}(K_{|X|}) \otimes C^{\alpha}(K_{|X|}) \rightarrow \mathbb{C}$. It follows immediately that the cocycle satisfies the cyclic condition since $\phi_{n}(f,g)$ satisfies the cyclic condition for any $n \in \mathbb{Z}_{\geq 0}$. Accordingly, it remains to show that $\phi$ is a Hochschild cocycle. For $f$, $g$, $h \in C^{\alpha}(K_{|X|})$, we may write $b\phi(f, g, h)$ as
\begin{eqnarray}\nonumber
b\phi(f,g,h) &=& \phi(fg, h) - \phi(f, gh) + \phi(hf, g) \\ \nonumber
&=& \lim_{n \rightarrow \infty}\phi_{n}(fg, h) - \lim_{n \rightarrow \infty}\phi_{n}(f, gh) + \lim_{n \rightarrow \infty}\phi_{n}(hf, g) \\ \nonumber
&=& \lim_{n \rightarrow \infty}\Bigl(\phi_{n}(fg, h) -\phi_{n}(f, gh) + \phi_{n}(hf, g)\Bigr) \\ \nonumber
&=& \lim_{n \rightarrow \infty}b\phi_{n}(f, g, h).
\end{eqnarray}
Therefore, to prove that $b \phi(f, g, h) = 0$ is equivalent to prove that $  \lim_{n \rightarrow \infty}b\phi_{n}(f, g, h) = 0$. 
Using $\delta(\eta \smile \tau) = \delta \eta \smile \tau + (-1)^{\operatorname{deg}(\eta)} \eta \smile \delta \tau$, we have 
\begin{eqnarray} \nonumber
\phi_{n}(fg, h) &=& \Bigl(fg \smile \delta h - h \smile \delta(fg)\Bigr)(I_{n}) \\ \nonumber
&=& \Bigl((f \smile g \smile \delta h) - (h \smile \delta f \smile g) - (h \smile f \smile \delta g)\Bigr)(I_{n}). \nonumber
\end{eqnarray}
Similarly, 
$$\phi_{n}(f, gh)\ \ \ =\ \ \ \Bigl((f \smile \delta g \smile h) + (f \smile g \smile \delta h) - (g \smile h \smile \delta f)\Bigr)(I_{n}),$$
$$\phi_{n}(hf, g)\ \ \ =\ \ \ \Bigl((h \smile f \smile \delta g) - (g \smile \delta h \smile f) - (g \smile h \smile \delta f)\Bigr)(I_{n}).$$
Therefore, 
\begin{eqnarray}
b\phi_{n}(f, g , h) &=& - \Bigl((h \smile \delta f \smile g) + (f \smile \delta g \smile h) + (g \smile \delta h \smile f)\Bigr)(I_{n}).
\end{eqnarray}
Since 
\begin{eqnarray}\nonumber
- (h \smile \delta f \smile g)(I_{n}) &=& \Bigl((\delta h \smile f \smile g) + (h \smile f \smile \delta g) - (\delta(hfg))\Bigr)(I_{n}) \\ \nonumber
&=& \Bigl((\delta h \smile f \smile g) + (h \smile f \smile \delta g)\Bigr)(I_{n}), \nonumber
\end{eqnarray}
we have 
\begin{eqnarray} \nonumber
(4) &=& \Bigl((\delta h \smile f \smile g) + (h \smile f \smile \delta g) - (f \smile \delta g \smile h) - (g \smile \delta h \smile f)\Bigr)(I_{n}) \\ \nonumber
&=& \sum_{(x,y) \in \epsilon(I_{n})} \pm \Bigl((\delta h \smile f \smile g) + (h \smile f \smile \delta g) - (f \smile \delta g \smile h) - (g \smile \delta h \smile f)\Bigr)(x,y) \\ \nonumber
&=& \sum_{(x,y) \in \epsilon(I_{n})} \pm (h(y) - h(x)) (g(y) - g(x)) (f(y) - f(x)). \nonumber
\end{eqnarray}
Therefore
\begin{eqnarray} \nonumber
|b \phi_{n}(f, g, h)| &\leq& \sum_{(x,y) \in \epsilon(I_{n})} |h(y) - h(x)| \cdot |g(y) - g(x)| \cdot |f(y) - f(x)| \\ \nonumber
&=& \sum_{(x,y) \in \epsilon(I_{n})} c_{f} \cdot c_{g} \cdot c_{h} \cdot |x - y|^{3 \alpha} \\ \nonumber
&\leq& c_{f} \cdot c_{g} \cdot c_{h} \cdot d_{K_{|X|}}^{3 \alpha} \cdot M \sum_{(j_{1}, \cdots, j_{n}) \in S^{\times n}} r_{j_{1}}^{3 \alpha} \cdot \cdots \cdot r_{j_{n}}^{3 \alpha} \\ \nonumber
&=& c_{f} \cdot c_{g} \cdot c_{h} \cdot d_{K_{|X|}}^{3 \alpha} \cdot M \cdot (\sum_{j \in S} r_{j}^{3 \alpha})^{n} \\ \nonumber
& \rightarrow & 0, \ \ \ {\rm as}\ n \rightarrow \infty.
\end{eqnarray}
This completes the proof of (a).

We now prove (b). We note that we have the pairing 
$$HH_{1}(C^{\alpha}(K_{|X|})) \times HH^{1}(C^{\alpha}(K_{|X|})) \rightarrow \mathbb{C}.$$ As seen in Theorem 3.10, we know that $\phi(1 \otimes x) \neq 0$, and this completes the proof of (b).
\end{proof} 
\begin{remark}

The algebra of $\alpha$-H\"{o}lder continuous functions on a compact metric space admits a Banach topology and it turns out to be a Banach algebra. However, we do not know whether or not the cocycle of Theorem 3.11 is continuous in the sense of a map between Banach algebras.

\end{remark}

\subsection{Examples}
We examine the cyclic cocycle on some cellular self-similar sets. The spaces on which the cocycles are examined are the examples given in Section $2.2$. \\


$\bullet$ {\bf Sierpinski gasket}\\
The theorems in the Section 3.3 may be applied to the Sierpinski gasket $SG$, and the cyclic $1$-cocycle $\phi$ on $C^{\alpha}(SG)$ is well-defined for $2\alpha > \dim_{H}(SG) = \log_{2}3$. Moreover, the cocycle is non-trivial since $|X| \neq |X_{1}|$.

$\bullet$ {\bf pinwheel fractal}\\
Pinwheel fractal $PF$ may also be seen as a cellular self-similar set. The self-similar structure consists of $4$ similitudes whose ratios are $\frac{1}{\sqrt{5}}$, see also section $2$. The cyclic cocycle is well-defined if $2 \alpha > \dim_{H}(PF) = \log_{\sqrt{5}}4$ and non-trivial in $HC^{1}$.\\

$\bullet$ {\bf Infinite isolated Sierpinski gaskets}\\
The second example in Section $2.2$ is a cellular self-similar structure, and we denote by $ISG$ the resulting cellular self-similar set. The cyclic $1$-cocycle may be defined on the space, and the cocycle is non-trivial. From now, we also discuss the structure of $HC^{0}(C^{Lip}(ISG))$. 

By the self-similar structure of $ISG$, $  \pi_{0}(ISG) = \bigoplus_{p \in \mathbb{N}} \mathbb{Z}$, each of whose summands corresponds to a connected component $Y_{p}$ of $ISG$. Therefore $ISG$ may be written as 
$$ISG = \bigsqcup_{p \in \mathbb{N}} Y_{p}.$$
Then we have the canonical inclusion map 
$$in_{p} : Y_{p} \rightarrow \bigsqcup_{p \in \mathbb{N}} Y_{p} = ISG$$
for any $p \in \mathbb{N}$. We now fix a base point $y_{p} \in Y_{p}$ for each $p \in \mathbb{N}$, and define a cyclic $0$-cocycle $\psi_{p}$ of $C^{Lip}(Y_{p})$ by taking the value of $y_{p}$ for any $f \in C^{Lip}(Y_{p})$. Therefore, the canonical inclusion map $in_{p}$ induces the map of cyclic cohomology groups:
$$(in_{p})^{*} : HC^{0}(C^{Lip}(Y_{p})) \rightarrow HC^{0}(C^{lip}(ISG)).$$

We now let $P$ be a finite subset of $\mathbb{N}$ and assume that $\Psi_{P} = \sum_{p \in P} \alpha_{p}(in_{p})_{*}([\psi_{p}]) = 0$. We also define $c_{p} \in C^{Lip}(ISG)$ by
\[
  c_{p}(y) = \begin{cases}
    1, & y \in Y_{p} \\
    0, & otherwise.
  \end{cases}
\]
Then, for any $\tilde{p} \in P$, we have a pairing of the Hochschild homology group and the Hochschild cohomology group of $C^{Lip}(ISG)$:
$$0 = \langle \Psi_{P},\ c_{p} \rangle = \sum_{p \in P} \alpha_{p}(in_{p})_{*}([\psi_{p}])(c_{\tilde{p}}) = \alpha_{\tilde{p}},$$
and which means that the set $\{(in_{p})_{*}([\psi_{p}])\}_{p \in P}$ is a linearly independent set. Since this argument also works for any finite set $P$ of $\mathbb{N}$, we can conclude that $\{(in_{p})_{*}([\psi_{p}])\}_{p \in \mathbb{N}}$ forms a linearly independent set of $HC^{0}(C^{Lip}(ISG))$, and therefore $HC^{0}(C^{Lip}(ISG))$ contains $\bigoplus_{p \in \mathbb{N}} \mathbb{C}$ as a $\mathbb{C}$-vector space.

\subsection{Further work}
Strichartz introduces the notion of ``fractafold" \cite{strichartz1, strichartz2}, and on which he examines fractal versions of the classical theories, for example,  Hodge-de Rham theory, spectral theory, homotopy theory. In particular, the Laplacian on some kinds of self-similar sets has been extensively studied, and it is applied to various fields  \cite{barlow, kigami, strichartz1, strichartz2}. Here, we will give some examples of finite unions of cellular self-similar sets. 
\begin{center}
  \includegraphics[bb = 0 0 1057 292, width = 140mm]{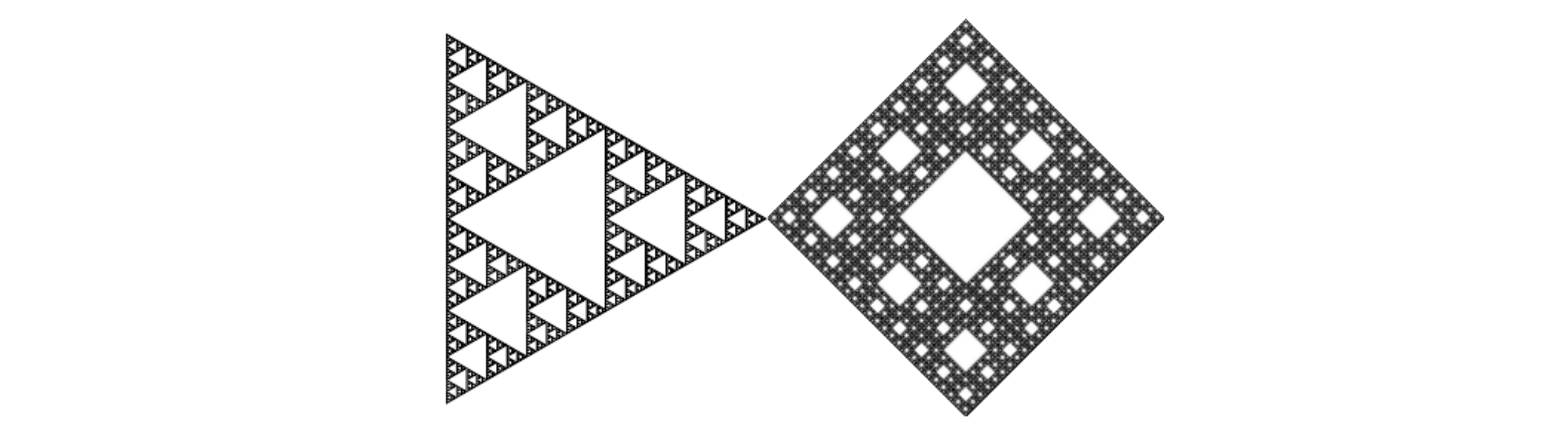}
\end{center}
The first example is the wedge sum of Sierpinski gasket and Sierpinski carpet with base points at their corners. Then the space is neither a cellular self-similar set nor a fractafold. However, the theorem may be applied to the space. Namely, the space is seen as the projective limit of the following spaces:
\begin{center}
  \includegraphics[bb = 0 0 1057 292, width = 140mm]{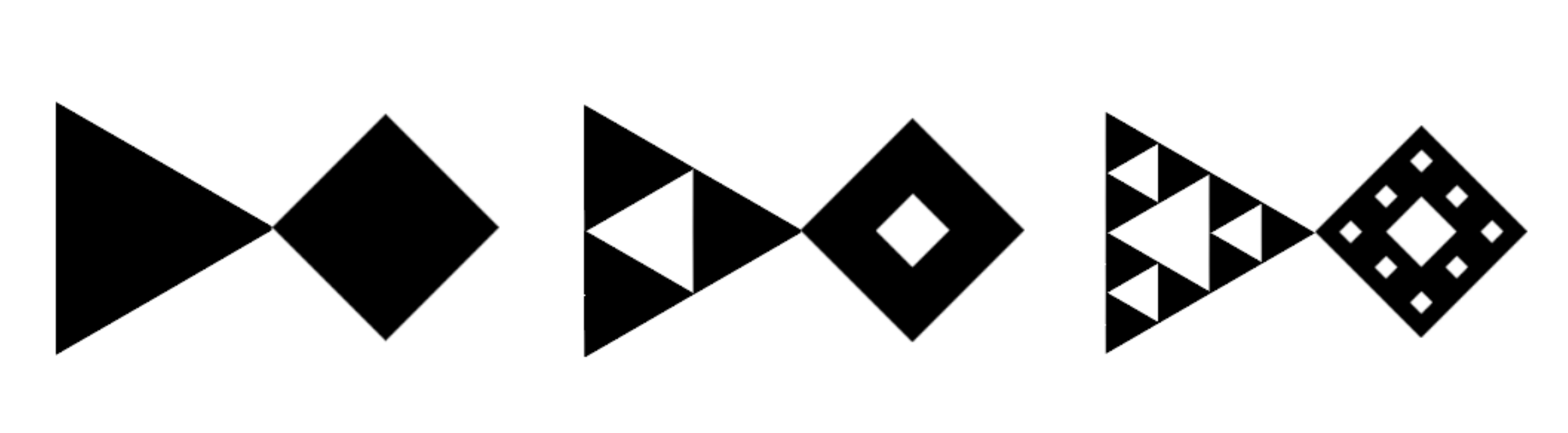}
\end{center}
The figure is obtained by taking the wedge sum of the sequences which give rise to Sierpinski gasket and Sierpinski carpet. Similarly, we have sequences of boundary chains $b_{0}$, $b_{1}$, $b_{2}$ and inner chains $I_{0}$, $I_{1}$, $I_{2}$ respectively:
\begin{center}
  \includegraphics[bb = 0 0 890 227, width = 140mm]{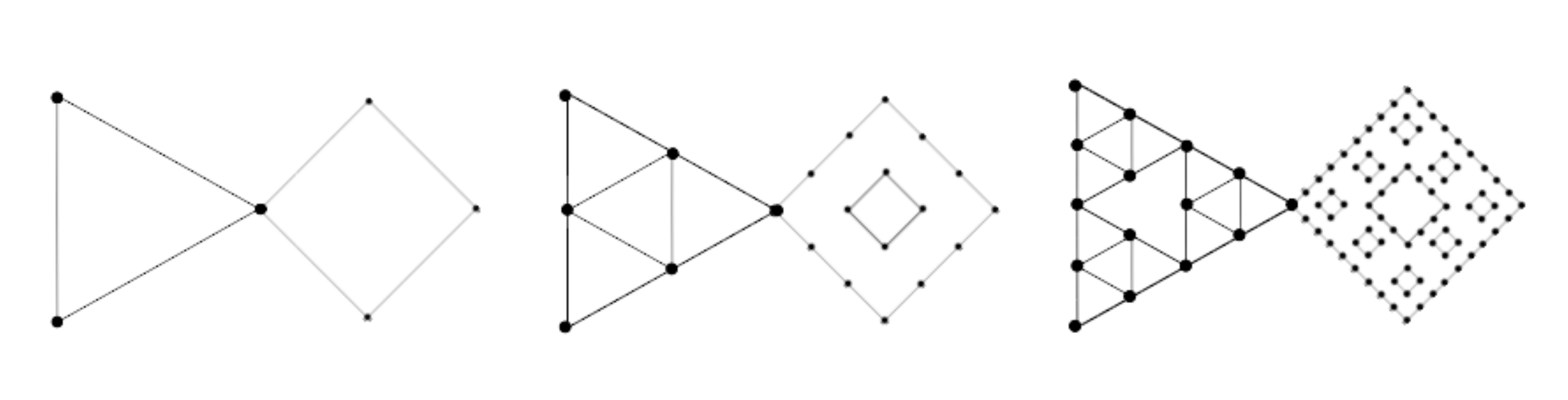}
  \includegraphics[bb = 0 0 890 227, width = 140mm]{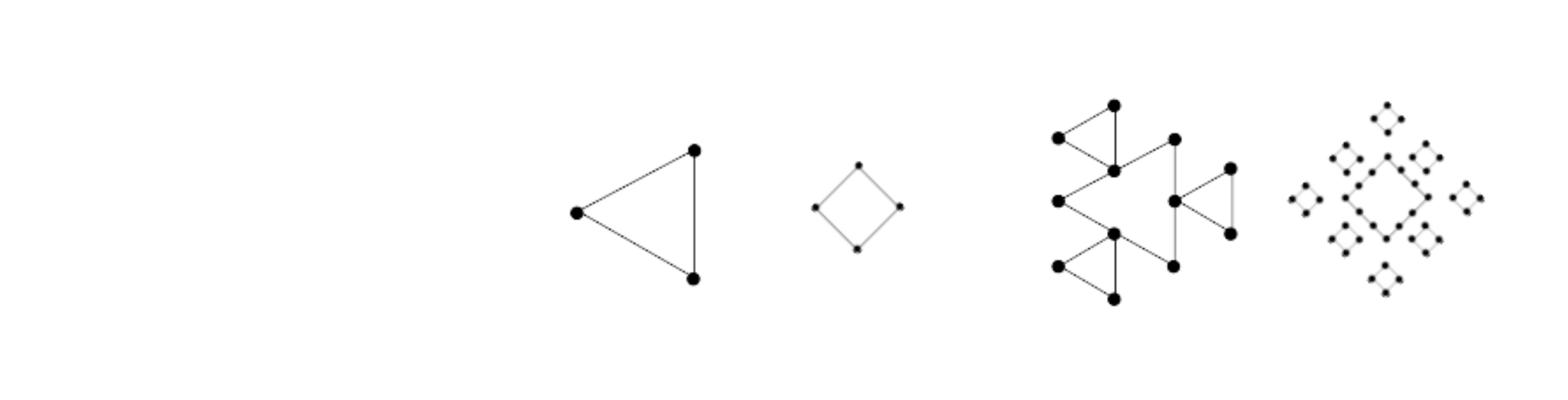}
\end{center}
We therefore have a cyclic quasi-$1$-cocycle, and the quasi-cocycle can be written by the element-wise sum of cyclic quasi-$1$-cocycles of $SG$ and $SC$. In order that that cyclic quasi-$1$-cocycle is a Cauchy sequence, it is enough that the H\"{o}lder index $\alpha$ satisfies the inequality $2 \alpha > \dim_{H}(SC)$. 

From the point of this view, $SG$ can be seen as a union of $3$ Sierpinski gaskets, and therefore $SG$ may be seen as a ``fractafold with boundary", see \cite{strichartz1, strichartz2} for details. As defined in the previous subsection, we have a cyclic cocycle on $SG$.

Finally, we will define a cyclic cocycle of the algebra of Lipschitz functions defined on a ``fractafold" based on the Sierpinski gasket:
\begin{center}
  \includegraphics[bb = 0 0 1057 292, width = 140mm]{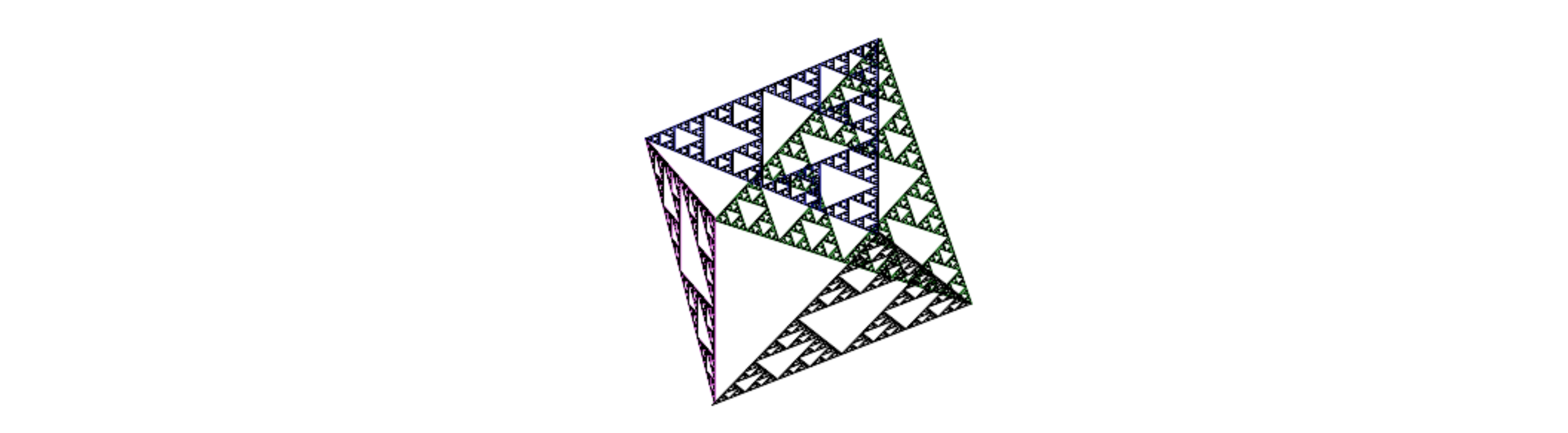}
\end{center}
The space is a union of four copies of Sierpinski gasket in $\mathbb{R}^{3}$ obtained by gluing the points at corners of a copy with each corner of the other Sierpinski gaskets. This space is one of the examples of what Strichartz calls ``fractafolds without boundaries", and we denote it by $FSG$. The space $FSG$ can be seen as the projective limit of a sequence of the spaces that is obtained by gluing copies of the sequence which gives rise to $SG$.
\begin{center}
  \includegraphics[bb = 0 0 1057 292, width = 140mm]{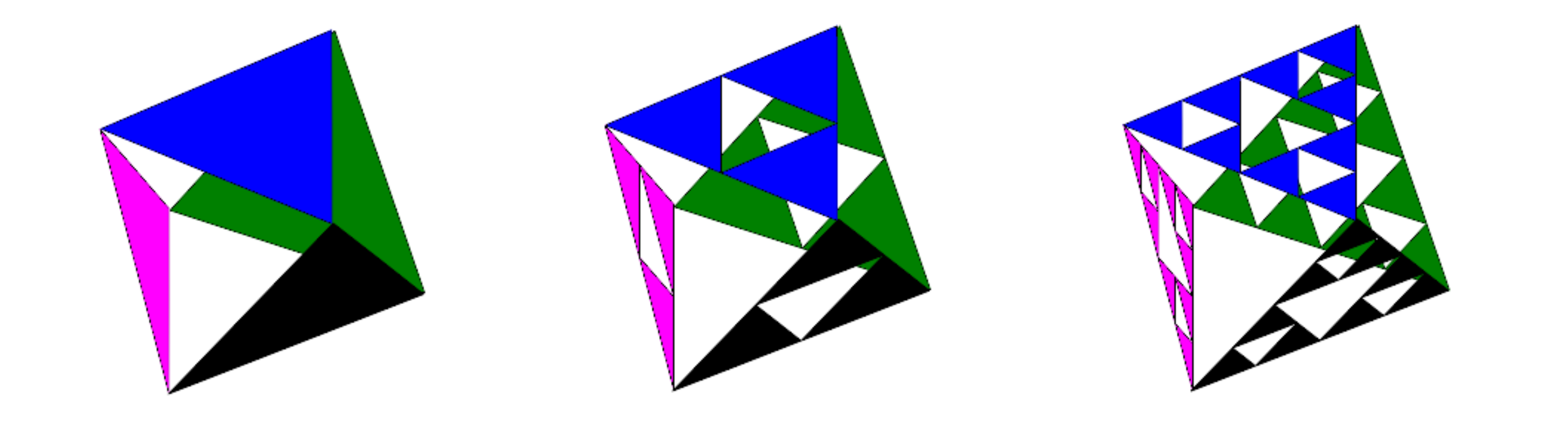}
\end{center}
We therefore get, by applying the theorem to each Sierpinski gasket, a cyclic $1$-cocycle on $C^{\alpha}(FSG)$ when $2 \alpha > \log_{2}3$. 

\begin{remark}
Strichartz introduces the Hodge-de Rham theory for fractal graphs \cite{acsy}. In this paper, Laplacian on some fractal sets are defined by exploiting the Alexander-Spanier cochain complexes. However, we do not know whether or not there exist any relation between the cyclic $1$-cocycle defined in the present paper and the Laplacian of \cite{acsy}.
\end{remark}

\bibliographystyle{te}
\bibliography{ref}


\Addresses

\end{document}